\documentclass[a4paper,11pt]{article}
\usepackage[latin1]{inputenc}
\usepackage[english]{babel}
\usepackage{amsmath}
\usepackage{amsfonts}
\usepackage{amssymb}
\usepackage{epsfig}
\usepackage{amsopn}
\usepackage{amsthm}
\usepackage{color}
\usepackage{graphicx}
\usepackage{subfigure}
\usepackage{enumerate}
\newtheorem{theorem}{Theorem}[section]

\newtheorem{lemma}[theorem]{Lemma}

\newtheorem*{theorem*}{Theorem}
\newtheorem*{lemma*}{Lemma}
\newtheorem*{remark*}{Remark}
\newtheorem*{definition*}{Definition}
\newtheorem*{proposition*}{Proposition}
\newtheorem*{corollary*}{Corollary}
\numberwithin{equation}{section}
\newcommand{\real}{\mathbb{R}}

\def\qed{\,\unskip\kern 6pt \penalty 500
\raise -2pt\hbox{\vrule \vbox to8pt{\hrule width 6pt
\vfill\hrule}\vrule}\par}
\definecolor{darkblue}{rgb}{0.05, .05, .65}
\definecolor{darkgreen}{rgb}{0.1, .65, .1}
\definecolor{darkred}{rgb}{0.8,0,0}
\newcommand{\beqn}{\begin{equation}}
\newcommand{\eeqn}{\end{equation}}
\newcommand{\bear}{\begin{eqnarray}}
\newcommand{\eear}{\end{eqnarray}}
\newcommand{\bean}{\begin{eqnarray*}}
\newcommand{\eean}{\end{eqnarray*}}
\begin{document}

\title{\huge \bf Existence of blow-up self-similar solutions for the supercritical quasilinear reaction-diffusion equation}

\author{
\Large Razvan Gabriel Iagar\,\footnote{Departamento de Matem\'{a}tica
Aplicada, Ciencia e Ingenieria de los Materiales y Tecnologia
Electr\'onica, Universidad Rey Juan Carlos, M\'{o}stoles,
28933, Madrid, Spain, \textit{e-mail:} razvan.iagar@urjc.es},\\
[4pt] \Large Ariel S\'{a}nchez,\footnote{Departamento de Matem\'{a}tica
Aplicada, Ciencia e Ingenieria de los Materiales y Tecnologia
Electr\'onica, Universidad Rey Juan Carlos, M\'{o}stoles,
28933, Madrid, Spain, \textit{e-mail:} ariel.sanchez@urjc.es}\\
[4pt] }
\date{}
\maketitle

\begin{abstract}
We establish the existence of self-similar solutions presenting finite time blow-up to the quasilinear reaction-diffusion equation
$$
u_t=\Delta u^m + u^p,
$$
posed in dimension $N\geq3$, $m>1$. More precisely, we show that there is always at least one solution in backward self-similar form if $p>p_s=m(N+2)/(N-2)$. In particular, this establishes \emph{non-optimality of the Lepin critical exponent} introduced in \cite{Le90} in the semilinear case $m=1$ and extended for $m>1$ in \cite{GV97, GV02}, for the existence of self-similar blow-up solutions. We also prove that there are multiple solutions in the same range, provided $N$ is sufficiently large. This is in strong contrast with the semilinear case, where the Lepin critical exponent has been proved to be optimal.
\end{abstract}

\

\noindent {\bf Mathematics Subject Classification 2020:} 35B33, 35B36, 35B44, 35C06, 35K57.

\smallskip

\noindent {\bf Keywords and phrases:} quasilinear reaction-diffusion, finite time blow-up, critical exponents, Lepin exponent, self-similar solutions.

\section{Introduction}

The quasilinear reaction-diffusion equation 
\begin{equation}\label{eq1}
u_t=\Delta u^m+u^p, 
\end{equation}
posed for $(x,t)\in\real^N$ and for $p>m>1$, together with its semilinear analogous
\begin{equation}\label{eq1.semi}
u_t=\Delta u+u^p, \qquad p>1,
\end{equation} 
have been strongly investigated in the last decades. Their main feature is the competition between the effect of the diffusion term, which conserves the initial mass of a data $u_0\in L^1(\real^N)$ while spreading it in the whole space $\real^N$, and the influence of the reaction term, which leads to an increasing $L^1$-norm of the solutions with respect to time and with the phenomenon of possible formation of singularities in a finite time $T\in(0,\infty)$. The latter phenomenon is known as \emph{finite time blow-up}: more precisely, we say that a solution $u$ to either \eqref{eq1} or \eqref{eq1.semi} blows up at time $T\in(0,\infty)$ if $u(t)\in L^{\infty}(\real^N)$ for any $t\in(0,T)$, but $u(T)\not\in L^{\infty}(\real^N)$. The famous paper by Fujita \cite{Fu66} established a critical exponent $p_F$ for the equation \eqref{eq1.semi}, called nowadays \emph{the Fujita exponent}, with the following property: for any $p\in(1,p_F]$, any non-trivial solution presents finite time blow-up, while for $p>p_F$, there are also global solutions (that is, solutions $u$ to \eqref{eq1.semi} such that $u(t)\in L^{\infty}(\real^N)$ for any $t>0$). This exponent has been generalized later in \cite{GKMS80, Qi93, Ga94, Ka95} to \eqref{eq1}, where it has the explicit value
$$
p_F=m+\frac{2}{N},
$$
and similar properties with respect to finite time blow-up of solutions. 

The study of finer properties related to finite time blow-up of solutions led to a new critical exponent, called \emph{the Sobolev critical exponent}, with the expression (also valid for $m=1$)
\begin{equation}\label{pSobolev}
p_s=\left\{\begin{array}{ll}\frac{m(N+2)}{N-2}, & {\rm if} \ N\geq3, \\ +\infty, & {\rm if} \ N\in\{1,2\}.\end{array}\right.
\end{equation}
In particular, in the semilinear case $m=1$, starting from the well-known works by Giga and Kohn \cite{GK85, GK87} and employing the technique of energy estimates in backward self-similar variables (see also \cite[Sections 22-25]{QS}), it has been shown that the blow-up rate of solutions to Eq. \eqref{eq1.semi} for $p\in(1,p_S)$ is always given by
\begin{equation}\label{BUrate}
\|u(x,t)\|_{\infty}\sim K(T-t)^{-1/(p-1)}, \qquad {\rm as} \ t\to T,
\end{equation}
which stems from the ordinary differential equation $u_t=u^p$, while the behavior as $t\to T$ of a solution $u$ to Eq. \eqref{eq1.semi} in a neighborhood of a blow-up point $a$ of $u$ is also characterized as follows:
$$
\lim\limits_{t\to T}(T-t)^{1/(p-1)}u(a+y\sqrt{T-t},t)=\pm(p-1)^{-1/(p-1)}
$$
uniformly in sets $|y|\leq C$, $C>0$ arbitrary (see \cite[Theorem 25.1]{QS}). Finite time blow-up with the rate \eqref{BUrate} has been named \emph{blow-up of type I}, while a finite time blow-up with any different rate is known as \emph{blow-up of type II}. The previous results have been extended to solutions to Eq. \eqref{eq1} at first in \cite{Ga86}, see also \cite[Chapter IV]{S4} and references therein. In particular, we can affirm that finite time blow-up in the subcritical range $m<p<p_s$ is always of type I. 

Studying the supercritical case $p>p_s$ is a much more complex task, as new phenomena related to the blow-up behavior of solutions to both \eqref{eq1} and \eqref{eq1.semi} may appear, in particular, blow-up of type II might occur for any $p>p_s$, as established first by Herrero and Vel\'azquez in \cite{HV94} and later classified in \cite{MM09, MM11} in the semilinear case, or more recently in \cite{MS21} for an even more general equation involving a variable coefficient. On the other hand, it has been shown that \emph{continuation after blow-up} may hold true in this range, that is, a new class of solutions that blow up at the origin at $t=T$ but afterwards they become again finite and non-singular has been identified. Such solutions, established in \cite{GV97}, have been named \emph{peaking solutions}. Related to them, two new critical exponents have been introduced. On the one hand, the \emph{Joseph-Lundgren exponent}
\begin{equation}\label{pJL}
p_{JL}=\left\{\begin{array}{ll}m\left(1+\frac{4}{N-4-2\sqrt{N-1}}\right), & {\rm if} \ N\geq11,\\[1mm] +\infty, & {\rm if} \ N<11,\end{array}\right.
\end{equation}
appeared first in the study of related semilinear elliptic problems in \cite{JL73} (see also \cite[Chapter 9]{QS}). On the other hand, the more tedious \emph{Lepin exponent} given by
\begin{equation}\label{Lepin}
p_L=\left\{\begin{array}{ll}1+\frac{3m+\sqrt{(m-1)^2(N-10)^2+2(m-1)(5m-4)(N-10)+9m^2}}{N-10}, & {\rm if} \ N\geq11, \\[1mm] +\infty, & {\rm if} \ N<11,\end{array}\right.
\end{equation}
appeared for the first time in \cite{Le90} related to the semilinear equation \eqref{eq1.semi} and was extended in \cite{GV97} to \eqref{eq1}. In particular, it is established that, in radial symmetry, blow-up is still of type I for $p\in(p_s,p_{JL})$ for solutions to Eq. \eqref{eq1.semi}, see for example \cite[Theorem 23.10]{QS}. More features related to the two critical exponents $p_{JL}$ and $p_L$ will be understood at the level of backward self-similar solutions in the next paragraphs. 

\medskip 

\noindent \textbf{Backward self-similar solutions}. We introduce next our main object of interest in this work, that is, solutions in backward radially symmetric self-similar form to Eq. \eqref{eq1}. Such solutions have the particular form
\begin{equation}\label{SSS}
u(x,t)=(T-t)^{-1/(p-1)}f(\xi), \qquad \xi=|x|(T-t)^{\beta}, \qquad \beta=\frac{p-m}{2(p-1)}>0,
\end{equation}
and the function $f(\cdot)$ will be called the \emph{self-similar profile} of the solution $u$. Plugging the ansatz \eqref{SSS} into Eq. \eqref{eq1}, we find that the self-similar profiles solve the differential equation
\begin{equation}\label{SSODE}
(f^m)''(\xi)+\frac{N-1}{\xi}(f^m)'(\xi)-\frac{1}{p-1}f(\xi)-\frac{p-m}{2(p-1)}\xi f'(\xi)+f^p(\xi)=0,
\end{equation}
together with the initial conditions induced by the radial symmetry 
\begin{equation}\label{init}
f(0)=A>0, \qquad f'(0)=0.
\end{equation}
Since it has been noticed that blow-up of type I has a self-similar behavior, identifying and classifying backward self-similar solutions became a fundamental question in understanding how general solutions presenting finite time blow-up behave as $t\to T$. Regarding the semilinear case \eqref{eq1.semi}, this classification is nowadays completely known. Thus, it has been established in \cite{GK85} (see also \cite{FT00}, where the non-existence is extended to more general equations involving also a variable coefficient) that for $1<p<p_s$ there are no backward self-similar solutions to Eq. \eqref{eq1.semi}, except for the one with constant profile
\begin{equation}\label{const.sol}
U_{*}(x,t)=k_*(T-t)^{-1/(p-1)}, \quad k_*=\left(\frac{1}{p-1}\right)^{1/(p-1)}.
\end{equation}
This is in strong contrast with the quasilinear equation \eqref{eq1}. Indeed, the existence of at least one backward self-similar solution to Eq. \eqref{eq1} in the subcritical range $m<p<p_s$ has been established in \cite[Theorem 4, p. 197, Chapter IV]{S4}. In the limiting case $p=p_s$, some stationary solutions have been identified in \cite[Section 7]{GV97}, see also \cite[Section 9]{IMS23b}, such solutions being valid for any $m\geq1$. Entering the range $p>p_s$, a stationary solution presenting a vertical asymptote at $x=0$, namely
\begin{equation}\label{stat.sol}
U_s(x)=c_s|x|^{-2/(p-m)}, \qquad c_s=\left[\frac{2m[p(N-2)-mN]}{(p-m)^2}\right]^{1/(p-m)},
\end{equation}
plays a very important role in the analysis. Indeed, for $m=1$, Lepin in \cite{Le88, Le90} classified self-similar solutions to Eq. \eqref{eq1.semi} with respect to the number of intersections of their graph with $U_s$. It was thus proved that, for $p\in(p_s,p_{JL})$, there are infinitely many backward radially symmetric self-similar solutions, while for $p\in[p_{JL},p_L)$ there is at least one, and in fact the critical exponent $p_L$ has been identified as the maximal value of $p$ for which the solution of the linearized equation in a neighborhood of $U_s$ has at least three zeros, leading to at least one oscillation of the actual self-similar profile with respect to the profile $U_s$. This approach has been extended to self-similar solutions to Eq. \eqref{eq1} by Galaktionov and V\'azquez in \cite[Section 12]{GV97}, where the existence of infinitely many \emph{peaking solutions} for $p\in(p_s,p_{JL})$ and of at least one for $p\in[p_{JL},p_L)$ is proved. Moreover, the existence of self-similar solutions with complete blow-up is also established in \cite[Section 14]{GV97} provided $p\in[p_s,p_{JL})$. All the self-similar profiles of solutions in this paragraph decay as $\xi\to\infty$ with a rate 
\begin{equation}\label{decay}
f(\xi)\sim C\xi^{-2/(p-m)}, \qquad C>0 \ {\rm arbitrary}.
\end{equation}

Thus, a question concerning the \emph{maximality of the critical exponent} $p_L$ for the existence of backward self-similar solutions in the form \eqref{SSS} to either Eq. \eqref{eq1} or Eq. \eqref{eq1.semi} appeared in a natural way, since the number $p_L$ itself has been obtained as a maximal reaction exponent allowing for at least one oscillation of self-similar profiles with respect to the stationary profile $U_s$. In a series of papers \cite{Mi04, Mi09, Mi10}, Mizoguchi established this maximality in the semilinear case \eqref{eq1.semi}, proving that indeed, for $p\geq p_L$ there are no self-similar solutions in the form \eqref{SSS} with profiles satisfying \eqref{SSODE} and the symmetry condition \eqref{init}. With respect to the quasilinear case of Eq. \eqref{eq1}, this question has been discussed in \cite{GV02} and remained open. 

Thus, our main goal is to show that, strongly departing from the semilinear case, the exponent $p_L$ \emph{is not maximal} for Eq. \eqref{eq1}, by constructing self-similar solutions for any $p>p_L$, and even multiple ones provided dimension $N$ is sufficiently large. But let us describe in detail our main results in the next paragraph.

\medskip 

\noindent \textbf{Main results.} As we explained above, the main result of this work is to establish existence of self-similar solutions in the backward form \eqref{SSS} for large exponents $p$, overpassing the Lepin critical exponent given in \eqref{Lepin}. We state the precise result below, where, for the sake of completeness, we include any $p>m$.
\begin{theorem}\label{th.1}
Let $N\geq3$.

\medskip

(a) For any $p>m$, there exists at least one solution in self-similar form \eqref{SSS} to Eq. \eqref{eq1}, which moreover has a decreasing profile with local behavior \eqref{decay} for some $C>c_s$.

\medskip 

(b) Given any natural number $K\geq2$, for any $p\in(m,(Km-1)/(K-1))$ there exist at least $K$ different solutions in self-similar form \eqref{SSS} to Eq. \eqref{eq1}, each of them having a different number of local maximum points. In particular, for any dimension 
\begin{equation}\label{LepK}
N>\frac{2(8K^2m-3Km-4K^2-2K+1)}{(2K-1)(m-1)}, \qquad K\geq2,
\end{equation}
there are at least $K$ different solutions in the form \eqref{SSS} for any $p\in(p_L,(Km-1)/(K-1))$. The self-similar profiles of all these solutions decay as in \eqref{decay} as $\xi\to\infty$. 
\end{theorem}

We mention here that, in the rest of the paper, we will \emph{only focus on exponents $p>p_s$}, since for $m<p\leq p_s$ the existence part of the result of Theorem \ref{th.1} is already known by \cite{GV97, S4}, while the multiplicity will be established in a more general framework in the future work \cite{IS24}.

Observe that the Lepin exponent $p_L$ (as well as the other critical exponents introduced above, $p_s$ and $p_{JL}$) can be regarded as a function of the dimension $N$. Thus, the first point that made us suspect that there can be self-similar solutions for $p>p_L$ was to notice that $p_L\to m$ as $N\to\infty$, and we have shown in our previous work \cite{IS22} that, for $p=m$, both existence and multiplicity of solutions hold true. Moreover, the mere deduction of the Lepin critical exponent is based on counting the number of oscillations of a profile with respect to the stationary solution $U_s$ (which counts as a profile itself). But it has been noticed in \cite[Chapter IV, pp. 191-193]{S4} that, for $m>1$, there is a new source of oscillations in the differential equation \eqref{SSODE}, with respect to the (constant) profile of the solution $U_{*}$, and we will exploit this fact in the proof of Theorem \ref{th.1}.

With respect to the technique in the proofs, instead of performing a shooting method directly in the equation \eqref{SSODE}, our approach is based on transforming this differential equation into an autonomous dynamical system and then perform a shooting on a two-dimensional unstable manifold (seen as a one-parameter family) in the phase space of this dynamical system. This transformation is performed in Section \ref{sec.syst}, where all the possible local behaviors of self-similar profiles $f(\xi)$ are classified and a number of preparatory lemmas concerning the extremal orbits of the manifolds involved in the analysis are states and proved. The proof of Theorem \ref{th.1} is then given in Section \ref{sec.main}, together with some numerical experiments supporting and helping to visualize the proof, see Figures \ref{fig1} and \ref{fig2}.

\section{A dynamical system. Analysis of critical points}\label{sec.syst}

Assume from now on that $N\geq3$ and $p>p_s$. Starting from generic profiles $f$ of self-similar solutions to Eq. \eqref{eq1} in the form \eqref{SSS}, we define the following new variables:
\begin{equation}\label{PSchange}
X(\eta)=\frac{1}{m(p-1)}\xi^2f^{1-m}(\xi), \qquad Y(\eta)=\frac{\xi f'(\xi)}{f(\xi)}, \qquad Z(\eta)=\frac{1}{m}\xi^{2}f^{p-m}(\xi),
\end{equation}
together with the independent variable $\eta=\ln\,\xi$. We deduce by direct calculation that \eqref{SSODE} is mapped by \eqref{PSchange} into the system
\begin{equation}\label{PSsyst}
\left\{\begin{array}{ll}\dot{X}=X(2-(m-1)Y), \\ \dot{Y}=X-(N-2)Y-Z-mY^2+\frac{p-m}{2}XY, \\ \dot{Z}=Z(2+(p-m)Y).\end{array}\right.
\end{equation}
Since we are looking only for non-negative self-similar solutions, we work from now on in the quadrant $X\geq0$, $Z\geq0$, observing that the planes $\{X=0\}$ and $\{Z=0\}$ are invariant. Thus, $Y$ is the only variable that might change sign, and according to its definition in \eqref{PSchange}, such a change of sign corresponds to a critical point of the profile $f(\xi)$. The main idea leading us to the proof of Theorem \ref{th.1} is to establish the existence of some specific trajectories of the system \eqref{PSsyst}, selected according to the local behaviors at both ends $\xi\to0$ and $\xi\to\infty$ when undoing the transformation \eqref{PSchange}. The system \eqref{PSsyst} has four critical points
\begin{equation*}
\begin{split}
P_0=(0,0,0), \ &P_1=\left(0,-\frac{N-2}{m},0\right), \ P_2=\left(0,-\frac{2}{p-m},\frac{2[p(N-2)-mN]}{(p-m)^2}\right),\\
&P_3=\left(\frac{2(mN-N+2)}{(m-1)(p-1)},\frac{2}{m-1},0\right),
\end{split}
\end{equation*}

\subsection{Local analysis}\label{subsec.local} 

We analyze in this section the critical points listed above, according to their stability.
\begin{lemma}[Local analysis near $P_0$]\label{lem.P0}
The critical point $P_0$ is a saddle point, presenting a two-dimensional unstable manifold and a one-dimensional stable manifold. The profiles corresponding to the orbits contained in the unstable manifold have the local behavior as in \eqref{init}, while the stable manifold is fully contained in the invariant $Y$ axis.
\end{lemma}
\begin{proof}
The linearization of the system \eqref{PSsyst} in a neighborhood of the point $P_0$ has the matrix
$$
M(P_0)=\left(
         \begin{array}{ccc}
           2 & 0 & 0 \\
           1 & -(N-2) & -1 \\
           0 & 0 & 2 \\
         \end{array}
       \right),
$$
with two positive eigenvalues $\lambda_1=\lambda_3=2$ and one negative eigenvalue $\lambda_2=-(N-2)$ and corresponding eigenvectors
\begin{equation}\label{eigen.P0}
e_1=(N,1,0), \qquad e_2=(0,1,0), \qquad e_3=(0,1,-N).
\end{equation} 
Since the eigenvector $e_2$ is contained in the $Y$-axis, which is invariant for the system \eqref{PSsyst}, the uniqueness of the stable manifold \cite[Theorem 3.2.1]{GH} implies that it is contained completely in this $Y$-axis. On the contrary, the unstable manifold is tangent to the plane spanned by the eigenvectors $e_1$ and $e_3$. We next wish to undo the change of variable \eqref{PSchange} on the unstable manifold. Remark that the first and third equations of the system \eqref{PSsyst} give
\begin{equation}\label{interm1}
Z(\eta)\sim CX(\eta), \qquad C\geq0, \qquad {\rm as} \ \eta\to-\infty,
\end{equation}
Plugging \eqref{interm1} into the second equation of the system \eqref{PSsyst}, neglecting the quadratic terms and observing that, at least in a neighborhood of $P_0$, one can invert $X(\eta)$ and thus define $Y$ as a function of $X$, we get by integration the family of orbits
\begin{equation}\label{interm2}
(l_C): \quad Y(\eta)\sim\frac{X(\eta)(1-C)}{N}, \qquad {\rm as} \ \eta\to-\infty,
\end{equation}
where $C$ is the non-negative constant coming from \eqref{interm1}. We will use from now on the notation $(l_C)_{C>0}$ to index as a one-parameter family the orbits on the unstable manifold of $P_0$, together with the limiting orbits $l_0$ contained in the invariant plane $\{Z=0\}$ and tangent to $e_1$, respectively $l_{\infty}$ contained in the invariant plane $\{X=0\}$ and tangent to $e_3$. Undoing the change of variable \eqref{PSchange}, we infer on the one hand from \eqref{interm1} that the orbits $l_C$ contain profiles with $f(0)=(C/(p-1))^{1/(p-1)}>0$ for $C>0$. On the other hand, \eqref{interm2} gives the more precise local behavior 
\begin{equation}\label{beh.P0}
f(\xi)\sim\left[\left(\frac{C}{p-1}\right)^{(m-1)/(p-1)}+\frac{(1-C)(m-1)}{2mN(p-1)}\xi^2\right]^{1/(m-1)}, \qquad {\rm as} \ \xi\to0.
\end{equation}
\end{proof}

\noindent \textbf{Remark.} We observe that the orbits $l_C$ with $C\in(0,1)$ go out of $P_0$ into the positive half-space $\{Y>0\}$ and the corresponding profiles in \eqref{beh.P0} are increasing in a right neighborhood of $\xi=0$. On the contrary, the orbits $l_C$ with $C>1$ go out into the negative half-space $\{Y<0\}$ and the corresponding profiles in \eqref{beh.P0} are decreasing in a right neighborhood of $\xi=0$. The orbit $l_1$, that is, $X(\eta)=Z(\eta)$ for any $\eta\in\real$, corresponds to the constant profile of the explicit solution $U_{*}$ defined in \eqref{const.sol}.

\begin{lemma}[Local analysis near $P_1$]\label{lem.P1}
The critical point $P_1$ is a saddle point with a two-dimensional unstable manifold fully contained in the invariant plane $\{Z=0\}$ and a one-dimensional stable manifold fully contained in the invariant plane $\{X=0\}$. There are no profiles contained in these orbits.
\end{lemma}
\begin{proof}
The linearization of the system \eqref{PSsyst} in a neighborhood of $P_1$ has the matrix
$$
M(P_1)=\left(
         \begin{array}{ccc}
           \frac{mN-N+2}{m} & 0 & 0 \\[1mm]
           \frac{mN-p(N-2)}{2m} & N-2 & -1 \\[1mm]
           0 & 0 & \frac{mN-p(N-2)}{m} \\
         \end{array}
       \right),
$$
with eigenvalues
$$
\lambda_1=\frac{mN-N+2}{m}>0, \qquad \lambda_2=N-2>0, \qquad \lambda_3=\frac{mN-p(N-2)}{m}<0,
$$
the latter being in force since $p>p_s$. It is then immediate to see that the eigenvectors corresponding to the eigenvalues $\lambda_1$ and $\lambda_2$ are both contained in the plane $\{Z=0\}$ (the second lying actually on the $Y$ axis), while the eigenvector corresponding to $\lambda_3$ is contained in the plane $\{X=0\}$. The invariance of these two planes for the system \eqref{PSsyst}, together with the uniqueness of the stable and unstable manifolds of a hyperbolic point \cite[Theorem 3.2.1]{GH} complete the proof.
\end{proof}
When analyzing the critical point $P_2$, we notice that the Joseph-Lundgren exponent $p_{JL}$ comes into play as follows.
\begin{lemma}[Local analysis near $P_2$]\label{lem.P2}
The critical point $P_2$ is a saddle-focus for $p\in(p_s,p_{JL})$ and a saddle point for $p\geq p_{JL}$, having in both cases a two-dimensional stable manifold fully contained in the invariant plane $\{X=0\}$ and a one-dimensional unstable manifold. The unique orbit contained in the unstable manifold corresponds to the stationary solution \eqref{stat.sol}.
\end{lemma}
\begin{proof}
The linearization of the system \eqref{PSsyst} in a neighborhood of the point $P_2$ has the matrix
$$
M(P_2)=\left(
         \begin{array}{ccc}
           \frac{2(p-1)}{p-m} & 0 & 0 \\[1mm]
           0 & -\frac{p(N-2)-m(N+2)}{p-m} & -1 \\[1mm]
           0 & \frac{2(p(N-2)-mN)}{p-m} & 0 \\
         \end{array}
       \right),
$$
whose eigenvalues satisfy
$$
\lambda_1=\frac{L}{p-m}>0, \ \lambda_2+\lambda_3=-\frac{p(N-2)-m(N+2)}{p-m}, \ \lambda_2\lambda_3=\frac{2(p(N-2)-mN)}{p-m}.
$$
Since $p>p_s$, we notice that $p(N-2)>m(N+2)>mN$, whence $\lambda_2+\lambda_3<0$ and $\lambda_2\lambda_3>0$. Hence, either $\lambda_2$ and $\lambda_3$ are both negative real numbers, or they are conjugate complex numbers with negative real part. This distinction (leading to either a saddle-focus or a saddle point) depends on the sign of the expression
$$
F(m,N,p)=(\lambda_2+\lambda_3)^2-4\lambda_2\lambda_3=\frac{(N-2)(N-10)p^2-2m(N^2-8N+4)p+m^2(N-2)^2}{(p-m)^2}.
$$ 
One can readily check by direct computation that 
$$
F(m,N,p_s)=-4(N-2)<0, \qquad F(m,N,p_{JL})=0
$$
and that numerator of $F(m,N,p)$ is a second degree polynomial with respect to the variable $p$, with positive dominating coefficient for $N\geq11$ and roots $p_{JL}$ and another root smaller than $p_s$. This implies that $F(m,N,p)>0$ provided $p>p_{JL}$. We conclude that $P_2$ is a saddle-focus for $p<p_{JL}$ and a (hyperbolic) saddle for $p\geq p_{JL}$. Noticing that the matrix $M(P_2)$ is a "block-matrix", it is easy to infer that the stable two-dimensional manifold is contained in the invariant plane $\{X=0\}$. The one-dimensional unstable manifold is unique according to \cite[Theorem 3.2.1]{GH} and thus contains a single orbit, with the property that 
$$
Z(\eta)\to\frac{2[p(N-2)-mN]}{(p-m)^2}, \qquad {\rm as} \ \eta\to-\infty.
$$
Noticing that $U_s$ satisfies the above limit, the uniqueness of the orbit leads to the identification claimed in the statement.
\end{proof}
We are left with the critical point $P_3$. 
\begin{lemma}[Local analysis near $P_3$]\label{lem.P3}
The critical point $P_3$ is either a saddle-focus or a saddle point having a two-dimensional stable manifold fully included in the invariant plane $\{Z=0\}$ and a one-dimensional unstable manifold which consists of a single orbit. On the profiles contained in this orbit we have $f(0)=0$, with the local behavior
\begin{equation}\label{beh.P3}
f(\xi)\sim\left[\frac{m-1}{2m(mN-N+2)}\right]^{1/(m-1)}\xi^{2/(m-1)}, \qquad {\rm as} \ \xi\to0.
\end{equation}
\end{lemma}
\begin{proof}
The linearization of the system \eqref{PSsyst} in a neighborhood of $P_3$ has the matrix
$$
M(P_3)=\left(
  \begin{array}{ccc}
    0 & -\frac{2(mN-N+2)}{p-1} & 0 \\[1mm]
    \frac{p-1}{m-1} & -\frac{N(m-1)^2+2(mp-1)}{(m-1)(p-1)} & -1 \\[1mm]
    0 & 0 & \frac{2(p-1)}{m-1} \\
  \end{array}
\right).
$$
The eigenvalues of it satisfy $\lambda_3=2(p-1)/(m-1)>0$ and
$$
\lambda_1\lambda_2=\frac{2(mN-N+2)}{m-1}>0, \qquad \lambda_1+\lambda_2=-\frac{N(m-1)^2+2(mp-1)}{(m-1)(p-1)}<0,
$$
hence the two eigenvalues $\lambda_1$ and $\lambda_2$ are either negative real numbers or complex conjugate numbers with negative real part. In both cases, the corresponding stable manifold, due to its uniqueness and the invariance of the plane $\{Z=0\}$, is fully contained in this plane. The unstable manifold of $P_3$ is one-dimensional and corresponds to the eigenvalue $\lambda_3$. On the unique trajectory of the system \eqref{PSsyst} contained in it, we have
$$
X(\eta)\to X(P_3):=\frac{2(mN-N+2)}{(m-1)(p-1)}, \qquad {\rm as} \ \eta\to-\infty,
$$
which readily gives \eqref{beh.P3} by undoing the transformation \eqref{PSchange}.
\end{proof}

\subsection{Analysis at infinity}\label{subsec.inf}

In order to complete the local analysis of the system \eqref{PSsyst}, we are left with studying its critical points at infinity. This is done by passing to the Poincar\'e hypersphere (according to, for example, \cite[Section 3.10]{Pe}) by setting
$$
X=\frac{\overline{X}}{W}, \qquad Y=\frac{\overline{Y}}{W}, \qquad Z=\frac{\overline{Z}}{W}.
$$
The critical points at infinity, expressed in these new variables, are thus found as the solutions of the system
\begin{equation}\label{Poincare}
\left\{\begin{array}{ll}\frac{1}{2}\overline{X}\overline{Y}[(p-m)\overline{X}-2\overline{Y}]=0,\\
(p-1)\overline{X}\overline{Z}\overline{Y}=0,\\
\frac{1}{2}\overline{Y}\overline{Z}[2p\overline{Y}-(p-m)\overline{X}]=0,\end{array}\right.
\end{equation}
to which the condition of belonging to the equator of the hypersphere, which gives $W=0$ and $\overline{X}^2+\overline{Y}^2+\overline{Z}^2=1$, has to be added, according to \cite[Theorem 4, Section 3.10]{Pe}. Since $\overline{X}\geq0$ and $\overline{Z}\geq0$, we readily find the following critical points at infinity:
\begin{equation}\label{crit.inf}
\begin{split}
&Q_1=(1,0,0,0), \ \ Q_{2,3}=(0,\pm1,0,0), \ \ Q_4=(0,0,1,0), \ \ Q_{\gamma}=\left(\gamma,0,\sqrt{1-\gamma^2},0\right),\\
&Q_5=\left(\frac{2}{\sqrt{4+(p-m)^2}},\frac{p-m}{\sqrt{4+(p-m)^2}},0,0\right),
\end{split}
\end{equation}
with $\gamma\in(0,1)$. In order to analyze the stability of the critical points whose $\overline{X}$ coordinate is nonzero, we project on the $X$ variable according to \cite[Theorem 5(a), Section 3.10]{Pe}, that is, performing the following change of variable
\begin{equation}\label{change2}
x=\frac{1}{X}, \qquad y=\frac{Y}{X}, \qquad z=\frac{Z}{X},
\end{equation}
in order to obtain the system 
\begin{equation}\label{PSinf1}
\left\{\begin{array}{ll}\dot{x}=x[(m-1)y-2x],\\
\dot{y}=-y^2+\frac{p-m}{2}y+x-Nxy-xz,\\
\dot{z}=(p-1)yz,\end{array}\right.
\end{equation}
where the derivatives are taken with respect to a new independent variable defined implicitly by the differential equation
\begin{equation}\label{indep2}
\frac{d\eta_1}{d\xi}=\frac{\alpha}{m}\xi f^{1-m}(\xi)=\frac{X(\xi)}{\xi}.
\end{equation} 
In the system \eqref{PSinf1}, the critical points $Q_1$, $Q_5$ and $Q_{\gamma}$ are identified as
$$
Q_1=(0,0,0), \ \ Q_5=\left(0,\frac{p-m}{2},0\right), \ \ Q_{\gamma}=(0,0,\kappa), \ \ {\rm with} \ \kappa=\frac{\sqrt{1-\gamma^2}}{\gamma}\in(0,\infty).
$$
\begin{lemma}[Local analysis near $Q_1$]\label{lem.Q1}
The critical point $Q_1$ has a two-dimensional center manifold and a one-dimensional unstable manifold which is contained in the $y$-axis. The direction of the flow on the center manifold is stable and the profiles contained in the orbits entering $Q_1$ on the center manifold present the local behavior \eqref{decay} as $\xi\to\infty$.
\end{lemma}
\begin{proof}
The linearization of the system \eqref{PSinf1} in a neighborhood of $Q_1$ has the matrix
$$M(Q_1)=\left(
         \begin{array}{ccc}
           0 & 0 & 0 \\[1mm]
           1 & \frac{p-m}{2} & 0 \\[1mm]
           0 & 0 & 0 \\
         \end{array}
       \right).
$$
The one-dimensional unstable manifold is contained in the $y$-axis, owing to its invariance and the eigenvector $e_2=(0,1,0)$ corresponding to the single positive eigenvalue. We analyze next the center manifolds of $Q_1$. This is done by further introducing the change of variable
\begin{equation*}
w=\frac{p-m}{2}y+x,
\end{equation*}
which maps the system \eqref{PSinf1} into the new system
\begin{equation}\label{PSinf1.bis}
\left\{\begin{array}{ll}\dot{x}=-\frac{2(p-1)}{p-m}x^2+\frac{2(m-1)}{p-m}xw,\\[1mm]
\dot{w}=\frac{p-m}{2}w-\frac{2}{p-m}w^2-\frac{p-m}{2}xz+\frac{2(m+1)-N(p-m)}{p-m}xw-\frac{2p-N(p-m)}{p-m}x^2,\\[1mm]
\dot{z}=-\frac{2(p-1)}{p-m}xz+\frac{2(p-1)}{p-m}zw,\end{array}\right.
\end{equation}
Looking for center manifolds in the general form 
$$
w(x,z)=ax^2+bxz+cz^2+O(|(x,z)|^3),
$$
according to \cite[Theorem 2, Section 2.5]{Carr}, we find that any center manifold of $Q_1$ has the equation
\begin{equation}\label{cmf}
\frac{p-m}{2}y+x=w=\frac{2[mN-p(N-2)]}{(p-m)^2}x^2+xz+o(|(x,z)|^2),
\end{equation}
while the flow on the center manifold is given by the reduced system (according to \cite[Theorem 2, Section 2.4]{Carr})
\begin{equation}\label{interm4}
\left\{\begin{array}{ll}\dot{x}&=-\frac{2(p-1)}{p-m}x^2+x^2O(|(x,z)|),\\[1mm]
\dot{z}&=-\frac{2(p-1)}{p-m}xz+xO(|(x,z)|^2),\end{array}\right.
\end{equation}
We thus deduce that the direction of the flow on the center manifolds is stable. The latter, together with the fact that the single nonzero eigenvalue is positive and \cite[Theorem 3.2']{Sij}, imply that the two-dimensional manifold is unique. On the center manifold, the reduced system \eqref{interm4} can be integrated in a first approximation to find
\begin{equation}\label{interm6}
\frac{z(\eta_1)}{x(\eta_1)}\to k>0, \qquad {\rm as} \ \eta_1\to\infty.
\end{equation}
We derive from the first equation in \eqref{interm4} and \eqref{change2} that, in a neighborhood of $Q_1$ and on the center manifold, we have
\begin{equation}\label{interm7}
x(\eta_1)\sim\frac{p-m}{2(p-1)\eta_1}, \qquad X(\eta_1)\sim\frac{2(p-1)}{p-m}\eta_1, \qquad {\rm as} \ \eta_1\to\infty,
\end{equation}
and furthermore, by replacing \eqref{interm7} into \eqref{indep2},
\begin{equation}\label{interm8}
\xi(\eta_1)=\int_0^{\eta_1}\frac{s}{X(s)}\,ds\to\int_{0}^{\infty}\frac{\eta_1}{X(\eta_1)}\,d\eta_1=+\infty,
\end{equation}
thus the behavior on orbits entering $Q_1$ corresponds to the limit $\xi\to\infty$ in terms of profiles. Taking this into account, by undoing \eqref{change2}, \eqref{interm6} leads to $Z(\xi)\to k$ as $\xi\to\infty$, which is equivalent to the local behavior \eqref{decay}, as claimed.
\end{proof}

\noindent \textbf{Remark.} In the original $(X,Y,Z)$ space, taking into account the local behavior \eqref{decay}, the orbits entering the critical point $Q_1$ have the properties $X(\eta)\to\infty$, $Y(\eta)\to-2/(p-m)$, $Z(\eta)\to k$ for any $k\in(0,\infty)$, all the previous limits being taken as $\eta\to\infty$. This representation will be useful in the sequel.

The next lemma deals with the stability of the critical point $Q_5$.
\begin{lemma}[Local analysis near $Q_5$]\label{lem.Q5}
The critical point $Q_5$ is a saddle point in the system \eqref{PSinf1}, with a two-dimensional unstable manifold and a one-dimensional stable manifold contained in the invariant $y$-axis. The orbits on the unstable manifold contain profiles presenting a ``dead-core" local behavior, in the sense that there exists $\xi_0\in(0,\infty)$ such that
\begin{equation}\label{beh.Q5}
f(\xi)=0, \ \xi\in[0,\xi_0], \qquad f(\xi)>0, \ \xi\in(\xi_0,\xi_0+\delta), \qquad (f^m)'(\xi_0)=0,
\end{equation}
for some $\delta>0$.
\end{lemma}
\begin{proof}
The linearization of the system \eqref{PSinf1} in a neighborhood of the critical point $Q_5$ has the matrix
$$
M(Q_5)=\left(
         \begin{array}{ccc}
           \frac{(m-1)(p-m)}{2} & 0 & 0 \\[1mm]
           1-\frac{N(p-m)}{2} & -\frac{p-m}{2} & 0 \\[1mm]
           0 & 0 & \frac{(p-m)(p-1)}{2} \\
         \end{array}
       \right),
$$
with eigenvalues 
$$
\lambda_1=\frac{(m-1)(p-m)}{2}>0, \quad \lambda_2=-\frac{p-m}{2}<0, \quad \lambda_3=\frac{(p-m)(p-1)}{2}>0.
$$
Since the eigenvector corresponding to $\lambda_2$ is $e_2=(0,1,0)$, the uniqueness of a stable manifold together with the invariance of the $y$ axis for the system \eqref{PSinf1} entail that the stable manifold is contained in the $y$ axis. On the unstable manifold, we have $y(\eta_1)\to(p-m)/2$ as $\eta_1\to-\infty$. We next plug this limit in the first equation of the system \eqref{PSinf1} and integrate to deduce that, in a first approximation, 
$$
X(\eta_1)\sim Ce^{-(m-1)(p-m)\eta_1/2}, \qquad {\rm as} \ \eta_1\to-\infty.
$$
The definition \eqref{indep2} then entails that
$$
\xi(\eta_1)\to\int_0^{-\infty}\frac{\eta_1}{X(\eta_1)}\,d\eta_1\sim-\int_{-\infty}^{0}C\eta_1e^{(m-1)(p-m)\eta_1/2}\,d\eta_1<\infty,
$$
whence $\xi(\eta_1)\to\xi_0\in(0,\infty)$ from the right, as $\eta_1\to-\infty$. With this in mind, the local behavior follows from the limit
$$
y(\eta_1)=\frac{Y(\eta_1)}{X(\eta_1)}\to\frac{p-m}{2}, \qquad {\rm as} \ \eta_1\to-\infty,
$$
which is equivalent, after undoing \eqref{PSchange}, to 
\begin{equation}\label{beh.Q5.prec}
f(\xi)\sim\left[\frac{(p-m)(m-1)}{4m(p-1)}(\xi^2-\xi_0^2)\right]_{+}^{1/(m-1)}, \qquad {\rm as} \ \xi\to\xi_0\in(0,\infty), \ \xi>\xi_0,
\end{equation}
satisfying \eqref{decay}, as claimed.
\end{proof}
Among the critical points that can be seen on the system \eqref{PSinf1}, we are left with the family $Q_{\gamma}=(0,0,\kappa)$, $\kappa\in(0,\infty)$.
\begin{lemma}[Local analysis near $Q_{\gamma}$]\label{lem.Qg}
For $\kappa_0=1$, and thus corresponding $\gamma_0=\sqrt{2}/2$, there exists a unique trajectory of the system \eqref{PSinf1} entering the critical point $Q_{\gamma_0}$ from the positive part of the phase space. This trajectory is contained in the plane $\{Y=0\}$ and corresponds to the unique constant profile giving rise to the solution $U_{*}$ given in \eqref{const.sol}. For any $\gamma\in(0,1)$ with $\gamma\neq\gamma_0$, all the center and unstable manifolds of $Q_{\gamma}$ are contained in the invariant plane $\{x=0\}$ and do not contain self-similar profiles.
\end{lemma}
\begin{proof}
For the reader's convenience, we give first a short, formal, but rather convincing argument in terms of profiles. Assume that there is a trajectory entering some point $(0,0,\kappa)$ in the system \eqref{PSinf1}, that is, $z(\eta_1)\to\kappa$ as $\eta_1\to\infty$, which gives
$$
\kappa=\lim\limits_{\xi\to\infty}\frac{Z(\xi)}{X(\xi)}=(p-1)\lim\limits_{\xi\to\infty}f^{p-1}(\xi),
$$
or equivalently
$$
\lim\limits_{\xi\to\infty}f(\xi)=\left(\frac{\kappa}{p-1}\right)^{1/(p-1)}=:L>0.
$$
By applying twice \cite[Lemma 2.9]{IL13} to the function $f-L$, respectively to the function $\xi f'$, we infer that there is a subsequence $(\xi_n)_{n\geq1}$ such that $\xi_n\to\infty$ and
$$
\lim\limits_{n\to\infty}\xi_nf'(\xi_n)=\lim\limits_{n\to\infty}\xi_n^2f''(\xi_n)=0.
$$
By evaluating the equation \eqref{SSODE} at $\xi=\xi_n$ and passing to the limit as $n\to\infty$, we are left only with the terms not containing derivatives, and thus we can identify $L=(1/(p-1))^{1/(p-1)}$ and deduce that necessarily $\kappa=1$. For the uniqueness of the orbit entering this point, we have to enter the local analysis of it in the system \eqref{PSinf1}. To this end, we translate the generic point $(0,0,\kappa)$ with any $\kappa>0$ to the origin by letting $z=\overline{z}+\kappa$ in \eqref{PSinf1} to find
\begin{equation}\label{interm8bis}
\left\{\begin{array}{ll}\dot{x}=x[(m-1)y-2x],\\
\dot{y}=-y^2+\frac{p-m}{2}y+(1-\kappa)x-Nxy-x\overline{z},\\
\dot{\overline{z}}=(p-1)\overline{z}y+(p-1)\kappa y,\end{array}\right.
\end{equation}
noticing that the linearization near the origin of the system \eqref{interm8bis} has a two-dimensional center manifold and an unstable manifold contained in the invariant plane $\{x=0\}$ (a fact that follows from the invariance of $\{x=0\}$, the uniqueness of the unstable manifold as given by \cite[Theorem 3.2.1]{GH} and the fact that the eigenspace of the eigenvalue $(p-m)/2$ is contained in the plane $\{x=0\}$). In order to study the center manifold, we proceed as in the proof of \cite[Lemma 2.4]{ILS23} and perform the double change of variables
\begin{equation}\label{interm9}
w=\frac{p-m}{2}y+(1-\kappa)x, \qquad v=\overline{z}-ly, \qquad l=\frac{2(p-1)\kappa}{p-m}.
\end{equation}
Replacing in \eqref{interm8bis} the variables $(x,y,\overline{z})$ by the new variables $(x,w,v)$, we obtain the system
\begin{equation}\label{interm10}
\left\{\begin{array}{ll}\dot{x}&=\left[\frac{2(m-1)(1-\kappa)}{p-m}-2\right]x^2+\frac{2(m-1)}{p-m}wx,\\
\dot{w}&=\frac{p-m}{2}w-\frac{2}{p-m}w^2-\frac{p-m}{2}vx-D_1x^2-D_2xw,\\
\dot{v}&=l(1-\kappa)x+D_3xv-D_4x^2-\frac{4lp}{(p-m)^2}w^2\\&-D_5xw+\frac{\alpha(p-1)}{\beta}vw,\end{array}\right.
\end{equation}
with coefficients
\begin{equation}\label{interm10b}
\begin{split}
&D_1=\frac{(1-\kappa)[(N-2)\beta+m\alpha(1-\kappa)-\alpha\kappa(p-1)]}{\beta}, \\ &D_2=\frac{N\beta+\alpha(m+1)-\alpha\kappa(m+p)}{\beta},\\
&D_3=\kappa+\frac{(p-1)\alpha(1-\kappa)}{\beta}, \\ &D_4=\frac{(1-\kappa)\alpha^2\kappa(p-1)[N\beta+\alpha(\kappa+p-2p\kappa)]}{\beta^3},\\
&D_5=\frac{\alpha^2\kappa(p-1)[N\beta+\alpha(\kappa+2p-3p\kappa)]}{\beta^3}.
\end{split}
\end{equation}
We next proceed as in the proof of \cite[Lemma 2.4]{ILS23} (to which we refer the reader for the calculation details) with the analysis of the center manifolds of the points $Q_{\gamma}$ identified as origin in the system \eqref{interm10}. In particular, following the local center manifold theory (see for example \cite[Section 2]{Carr}), we conclude that the center manifold has the form
$$
w=C_1x^2+C_2xv+C_3v^2+o(|(x,v)|^2),
$$
where the coefficients $C_1$, $C_2$ and $C_3$ can be explicitly calculated, and the flow on it is given by the reduced system
\begin{equation}\label{flowgamma}
\left\{\begin{array}{ll}\dot{x}&=\left[\frac{2(m-1)(1-\kappa)}{p-m}-2\right]x^2+x^2O(|(x,v)|),\\
\dot{v}&=\left[l(1-\kappa)\right]x+\left[-l+\frac{l(1-\kappa)}{\kappa}\right]xv\\&-Dx^2+xO(|(x,v)|^2),\end{array}\right.
\end{equation}
where
$$
D=\frac{l(1-\kappa)\alpha[(N-l)\beta+\alpha p(1-\kappa)]}{\beta^2}, \qquad l=\frac{2(p-1)\kappa}{p-m}.
$$
By performing the (implicit) change of independent variable $d\theta=xd\eta_1$ in \eqref{flowgamma} (which is equivalent to a simplification by $x$ in the right hand side of the system, which can be performed outside the plane $\{x=0\}$) and integrating the remaining system in a first approximation, it readily follows that when $\kappa\neq1$ there are no orbits connecting to $Q_{\gamma}$ from outside the invariant plane $\{x=0\}$. On the contrary, if $\gamma=\gamma_0$, that is, $\kappa=1$, the system \eqref{flowgamma} (taking derivatives with respect to the new variable $\theta$ defined implicitly by $d\theta=xd\eta_1$) becomes
\begin{equation}\label{flowgamma2}
\left\{\begin{array}{ll}\dot{x}=-2x+xO(|(x,x)|),\\
\dot{v}=\frac{2(p-1)}{p-m}v+O(|(x,v)|^2),\end{array}\right.
\end{equation}
and we notice that $(x,v)=(0,0)$ is a saddle point for the system \eqref{flowgamma2}. This fact, together with the theory in \cite[Section 3]{Sij}, prove the uniqueness of the orbit entering $Q_{\gamma_0}$ on its stable manifold, coming from the region $\{x>0\}$. The same argument as in \eqref{interm7} and \eqref{interm8} (not depending on the $z$ variable) prove that along this orbit, $\eta_1\to+\infty$ implies $\xi\to\infty$, and this uniqueness implies the identification in profiles of this unique orbit with the constant profile in \eqref{const.sol}. For more calculation details, the reader is referred to the proof of \cite[Lemma 2.4]{ILS23} (see also \cite[Lemma 2.3]{IMS23}).
\end{proof}
We pass now to the analysis of the system \eqref{PSsyst} near the critical points $Q_2$ and $Q_3$. According to \cite[Theorem 5(b), Section 3.10]{Pe}, we perform a different change of variable projecting on the dominating $Y$ variable 
\begin{equation}\label{change3}
x=\frac{X}{Y}, \qquad z=\frac{Z}{Y}, \qquad w=\frac{1}{Y},
\end{equation}
together with a new independent variable $\eta_2$ defined implicitly by
\begin{equation}\label{indep3}
\frac{d\eta_2}{d\xi}=\frac{Y(\xi)}{\xi},
\end{equation}
to get the following system
\begin{equation}\label{PSinf2}
\left\{\begin{array}{ll}\pm\dot{x}=-x-Nxw+\frac{p-m}{2}x^2+x^2w-xzw,\\[1mm]
\pm\dot{z}=-pz-Nzw+\frac{p-m}{2}xz+xzw-z^2w,\\[1mm]
\pm\dot{w}=-mw-(N-2)w^2+\frac{p-m}{2}xw+xw^2-zw^2,\end{array}\right.
\end{equation}
where the signs plus and minus correspond to the direction of the flow in a neighborhood of the points, that is, the minus sign applies to $Q_2$ and the plus sign applies to $Q_3$. In this system (taken with the corresponding signs), the critical points $Q_2$, respectively $Q_3$, are mapped into the origin $(x,z,w)=(0,0,0)$.
\begin{lemma}\label{lem.Q23}
The critical points $Q_2$ and $Q_3$ are, respectively, an unstable node and a stable node. The orbits going out of $Q_2$ correspond to profiles $f(\xi)$ such that there exists $\xi_0\in(0,\infty)$ and $\delta>0$ for which
\begin{equation}\label{beh.Q2}
f(\xi_0)=0, \qquad (f^m)'(\xi_0)=C>0, \qquad f>0 \ {\rm on} \ (\xi_0,\xi_0+\delta),
\end{equation}
while the orbits entering the stable node $Q_3$ correspond to profiles $f(\xi)$ such that there exists $\xi_0\in(0,\infty)$ and $\delta\in(0,\xi_0)$ for which
\begin{equation}\label{beh.Q3}
f(\xi_0)=0, \qquad (f^m)'(\xi_0)=-C<0, \qquad f>0 \ {\rm on} \ (\xi_0-\delta,\xi_0).
\end{equation}
\end{lemma}
\begin{proof}
It follows by direct calculation of the matrix of the linearization that the critical point $(0,0,0)$ is an unstable node in the system \eqref{PSinf2} taken with minus signs (which corresponds to $Q_2$ according to the theory in \cite[Theorem 5(b), Section 3.10]{Pe}) and a stable node in the system \eqref{PSinf2} taken with plus signs (which corresponds to $Q_3$). In order to establish the local behavior, we show that $\eta_2\to\pm\infty$ implies $\xi\to\xi_0\in(0,\infty)$. We perform this analysis for the stable node $Q_3$ (taking the plus signs in the left hand side of the system \eqref{PSinf2}), the analysis for $Q_2$ being completely similar. Let us observe that the last equation of the system \eqref{PSinf2} gives in a neighborhood of the origin,
$$
\dot{w}(\eta_2)\sim-mw(\eta_2), \qquad {\rm whence} \qquad Y(\eta_2)=\frac{1}{w(\eta_2)}\sim Ce^{m\eta_2},
$$
as $\eta_2\to\infty$. We thus deduce from \eqref{indep3} that
$$
\xi(\eta_2)=\int_0^{\eta_2}\frac{s}{Y(s)}\,ds\to\int_0^{\infty}\frac{\eta_2}{Y(\eta_2)}\,d\eta_2<\infty,
$$
as $\eta_2\to\infty$, since the improper integral of $C\eta_2e^{-m\eta_2}$ converges as $\eta_2\to\infty$. We infer that, on the orbits entering the stable node $Q_3$, the profiles are supported on $[0,\xi_0]$ for some $\xi_0\in(0,\infty)$ and the local behavior near the point $Q_3$ translates into the one as $\xi\to\xi_0$, $\xi<\xi_0$. The first and third equation of the system \eqref{PSinf2} give in a first approximation
$$
w(\eta_2)\sim Cx^m(\eta_2), \qquad {\rm as} \ \eta_2\to\infty,
$$
which in terms of the initial variables $(X,Y,Z)$ leads to
\begin{equation}\label{interm12}
\frac{1}{Y(\xi)}\sim\frac{C_1X(\xi)^m}{|Y(\xi)|^m}, \qquad {\rm as} \ \xi\to\xi_0, \ \xi<\xi_0,
\end{equation}
with $C_1<0$ a negative constant. In terms of profiles, we obtain from \eqref{interm12} and \eqref{PSchange} that
$$
f(\xi)^{m-1}f'(\xi)\sim C_2\xi^{(m+1)/(m-1)}, \qquad {\rm as} \ \xi\to\xi_0, \ \xi<\xi_0,
$$
for another generic constant $C_2<0$ (that can be explicitly expressed in terms of $C_1$). The latter leads to the local behavior \eqref{beh.Q3} as $\xi\to\xi_0\in(0,\infty)$. The proof of the local behavior \eqref{beh.Q2} is completely similar and will be omitted here.
\end{proof}
We are left with the local analysis of the critical point $Q_4$. The same technique as above does not work directly, since the projection on the $Z$ variable inspired by \cite[Theorem 5(c), Section 3.10]{Pe}) leads to a system with three zero eigenvalues at the origin, whose analysis requires deeper techniques. However, we shall show that there is no trajectory coming from the finite part of the phase space and entering $Q_4$ directly on \eqref{SSODE}. Indeed, on a trajectory entering $Q_4$, we have
$$
Z(\eta)\to\infty, \quad \frac{Y(\eta)}{Z(\eta)}\to0, \quad \frac{X(\eta)}{Z(\eta)}\to0, \quad {\rm as} \ \eta\to\infty,
$$
which in terms of profiles writes
\begin{equation}\label{interm13}
\xi^2f^{p-m}(\xi)\to\infty, \quad \frac{f'(\xi)}{\xi f^{p-m+1}(\xi)}\to0, \quad f^{1-p}(\xi)\to0, 
\end{equation}
either as $\xi\to\infty$ or as $\xi\to\xi_0\in(0,\infty)$, $\xi<\xi_0$. 
\begin{lemma}\label{lem.Q4}
There is no profile $f$ solution to \eqref{SSODE} for which \eqref{interm13} holds true.
\end{lemma}
\begin{proof}
The third limit in \eqref{interm13} gives that $f(\xi)\to\infty$, so that we either deal with a vertical asymptote or with an unbounded solution as $\xi\to\infty$. Assume for contradiction that there exists a profile $f$ with such property. We split the proof into two cases.

\medskip 

\noindent \textbf{Case 1. The limits in \eqref{interm13} are taken as $\xi\to\xi_0\in(0,\infty)$}. Notice first that $f$ is in this case increasing in a left neighborhood of $\xi_0$. Assume for contradiction that this is not the case, thus there is a sequence of points of local minima for $f$, $(\xi_n)_{n\geq1}$, such that $\xi_n\to\xi_0$ as $n\to\infty$. By evaluating \eqref{SSODE} at $\xi=\xi_n$, and taking into account that $f'(\xi_n)=(f^m)'(\xi_n)=0$, $(f^m)''(\xi_n)\geq0$, we find that 
$$
f^p(\xi_n)-\frac{1}{p-1}f(\xi_n)=f(\xi_n)\left[f^{p-1}(\xi_n)-\frac{1}{p-1}\right]\leq0,
$$
which is a contradiction for $n$ sufficiently large, as $f(\xi_n)\to\infty$ as $n\to\infty$. Hence $f$ is increasing on some neighborhood $(\xi_0-\delta,\xi_0)$ for some $\delta>0$. We then split \eqref{SSODE} for $\xi\in(\xi_0-\delta,\xi_0)$ as follows:
\begin{equation}\label{interm14}
(f^m)''(\xi)+\left[\frac{m(N-1)}{\xi}f^{m-1}(\xi)-\beta\xi\right]f'(\xi)+\left[f^{p-1}(\xi)-\frac{1}{p-1}\right]f(\xi)=0,
\end{equation}
where we recall that $\beta=(p-m)/2(p-1)$. Since both terms in brackets tend to $+\infty$ as $\xi\to\xi_0$ and $f'(\xi)\geq0$ for $\xi\in(\xi_0-\delta,\xi_0)$, it follows from \eqref{interm14} that $(f^m)''(\xi)\to-\infty$. But a simple calculus exercise shows that the latter is not possible when having a vertical asymptote, and we thus reach a contradiction with the existence of such profiles $f$ as in \eqref{interm13}.

\medskip 

\noindent \textbf{Case 2. The limits in \eqref{interm13} are taken as $\xi\to\infty$}. In exactly the same way as in Case 1, we get that $f$ is increasing on some interval $\xi\in(R,\infty)$, $R>0$. We then split \eqref{SSODE} in the following way: 
\begin{equation}\label{interm15}
(f^m)''(\xi)+\frac{N-1}{\xi}(f^m)'(\xi)+\left[\frac{1}{2}f^{p-1}(\xi)-\frac{1}{p-1}\right]f(\xi)+\frac{1}{2}f^{p}(\xi)-\beta\xi f'(\xi)=0.
\end{equation}
Since the second term in \eqref{interm15} is positive and the third term tends to $+\infty$ as $\xi\to\infty$, and by basic calculus considerations one cannot have $(f^m)''(\xi)\to-\infty$ when $f(\xi)\to\infty$ in an increasing way as $\xi\to\infty$, we infer that
$$
\frac{1}{2}f^p(\xi)-\beta\xi f'(\xi)\to-\infty, \quad {\rm as} \ \xi\to\infty.
$$
This gives in particular that there is $\xi_1\in(R,\infty)$ such that 
$$
\frac{1}{2}f^p(\xi)<\beta\xi f'(\xi), \quad {\rm or \ equivalently} \quad \frac{1}{2\beta\xi}<f^{-p}(\xi)f'(\xi),
$$
for any $\xi\in(\xi_1,\infty)$. By integrating over $(\xi_1,\xi)$, we obtain 
$$
\frac{1}{2\beta}\ln\left(\frac{\xi}{\xi_1}\right)<\frac{1}{1-p}\left(f(\xi)^{1-p}-f(\xi_1)^{1-p}\right),
$$
for any $\xi\in(\xi_1,\infty)$. But the latter leads to a contradiction by passing to the limit as $\xi\to\infty$, since the left hand side tends to $+\infty$, while the second one is bounded. We thus infer that there are no profiles satisfying \eqref{interm13} as $\xi\to\infty$, completing the proof.
\end{proof}

\subsection{Some preparatory results}\label{subsec.prep}

In this section, we establish some results related to the global behavior of orbits in the invariant planes $\{X=0\}$ of the system \eqref{PSsyst}, respectively $\{x=0\}$ of the system \eqref{PSinf1}, that will be very useful in the proof of Theorem \ref{th.1}. 
In the plane $\{X=0\}$, the system \eqref{PSsyst} reduces to the following one:
\begin{equation}\label{PSsystX0}
\left\{\begin{array}{ll}\dot{Y}=-(N-2)Y-Z-mY^2, \\ \dot{Z}=Z(2+(p-m)Y).\end{array}\right.
\end{equation}
\begin{lemma}\label{lem.X0}
The critical point $P_0=(0,0)$ is a saddle point for the system \eqref{PSsystX0}. The unique orbit going out of it on its unstable manifold connects to the critical point $P_2$.
\end{lemma}
\begin{proof}
The proof follows closely the one of \cite[Lemma 5.2]{IMS23b}, but we give some details for the sake of completeness. The fact that $P_0$ is a saddle is immediate by inspecting the matrix of the linearization. Let us consider now the following curve and its normal direction
\begin{equation}\label{cyl}
\frac{N}{N-2}Y(mY+N-2)+Z=0, \quad \overline{n}=\left(\frac{N}{N-2}(2mY+N-2),1\right), 
\end{equation}
with $Y\in(-(N-2)/m,0)$, and which connects $P_0$ to $P_1$. The direction of the flow of the system \eqref{PSsystX0} across this curve is given by the sign of
$$
H(Y):=\frac{N(p_s-p)}{N-2}Y^2(mY+N-2)<0, 
$$
since $p>p_s$. This implies that the region $\mathcal{R}$ in the plane $\{X=0\}$ limited by the $Y$-axis and the graph of the curve \eqref{cyl}, that is,
$$
\mathcal{R}:=\left\{(Y,Z): Y\in\left[-\frac{N-2}{m},0\right], Z\in\left[0,-\frac{N}{N-2}Y(mY+N-2)\right]\right\},
$$
is positively invariant. We show next that the unique orbit contained in the unstable manifold of $P_0$ enters the region $\mathcal{R}$ for $p>p_s$. We observe that both the curve \eqref{cyl} and the unstable manifold of $P_0$ are tangent in a first approximation to the vector $(-1,N)$, hence we have to go to a second order approximation of the trajectory going out of $P_0$ in order to spot its dependence on $p$. We thus look for a Taylor expansion of the form 
\begin{equation}\label{interm16}
Z=-NY+KY^2+o(Y^2),
\end{equation}
and (following for example \cite[Section 2.7]{Shilnikov}) we require that the flow of the system \eqref{PSsystX0} across the curve \eqref{interm16} is of the form $o(Y^2)$, taking as normal vector $\overline{n}=(N-2KY,1)$. By computing the scalar product of the vector field of the system \eqref{PSsystX0} with $\overline{n}$ and replacing $Z$ from \eqref{interm16}, we are left with the expression of the flow given by
$$
G(Y)=-(KN+Np+2K)Y+K(m+2K+p)Y^3+o(Y^2),
$$
so that, in order to have an approximation up to the second order, we need to cancel out the coefficient of the quadratic term, that is,
$$
K=-\frac{Np}{N+2}.
$$
We thus observe that the trajectory of the unique orbit in the unstable manifold of $P_0$ varies in a decreasing way as a function $Z=Z(Y)$ with respect to $p$ in a neighborhood of $P_0$. Since it coincides with \eqref{cyl} for $p=p_s$, we conclude that it points inside the region $\mathcal{R}$ for $p>p_s$ and thus remains forever in this region. Moreover, we also notice by direct calculation that $P_2\in\mathcal{R}$.

We prove next that the system \eqref{PSsystX0} does not admit limit cycles for $p>p_s$. Since it is a planar system, we employ Dulac's Criteria (see for example \cite[Theorem 2, Section 3.9]{Pe}). To this end, by multiplying the vector field of \eqref{PSsystX0} by the function $B(Y,Z)=Z^a$ for some $a$ to be chosen later, and taking the divergence of the new vector field, we get
\begin{equation}\label{interm17}
\frac{\partial(Z^a\dot{Y})}{\partial Y}+\frac{\partial(Z^a\dot{Z})}{\partial Z}=[2(a+1)-(N-2)]Z^a+[(p-m)(a+1)-2m]YZ^a.
\end{equation}
Choosing $a=(3m-p)/(p-m)$, we cancel out the second term in \eqref{interm17} and the previous divergence is given by
$$
h(Z):=\left(\frac{4m}{p-m}-N+2\right)Z^a<0,
$$
provided $p>p_s$ (notice that $h(Z)=0$ exactly for $p=p_s$). Thus, Dulac's Criteria and the invariance of the axis $\{Z=0\}$ show that there are no limit cycles in the region $\mathcal{R}$. The stability of the critical point $P_2$ together with the Poincar\'e-Bendixon's theory (see for example \cite[Section 3.7]{Pe}) conclude the proof.
\end{proof}

\noindent \textbf{Remark. Stationary solutions.} For $p=p_s$, the curve \eqref{cyl} is an exact solution of the system \eqref{PSsystX0}. Undoing the transformation in terms of profiles, we arrive to the following family of stationary solutions
\begin{equation}\label{sol.sobolev}
U_C(x)=\left[\frac{(N-2)(N+\sigma)C}{(|x|^{2}+C)^2}\right]^{(N-2)/4m}, \qquad C>0.
\end{equation}
The reader is referred to \cite[Section 9]{IMS23b} for a detailed proof (with $m<1$ in that case) of the deduction of the solutions \eqref{sol.sobolev}. However, we readily notice by direct calculation that they remain solutions to Eq. \eqref{eq1} for $p=p_s$ also when $m>1$.

\medskip 

We next go to the system \eqref{PSinf1} and we analyze the invariant plane $\{x=0\}$. Since this system decouples when letting $x=0$, let us first replace the variable $z$ in \eqref{PSinf1} by $w=xz\geq0$, a change of variable that has been already employed in \cite[Section 4]{IS21}. With respect to the finite part of the phase space associated to the system \eqref{PSsyst}, which corresponds to $x>0$, there is no essential variation involved by this change of variable. However, in the plane $\{x=0\}$ we make thus a zoom on the orbits and observe better their behavior. Performing this change and only then letting $x=0$, we are left with the system
\begin{equation}\label{PSsystw0}
\left\{\begin{array}{ll}\dot{y}=-y^2+\frac{p-m}{2}y-w,\\ \dot{w}=(m+p-2)yw,\end{array}\right.
\end{equation}
where derivatives are taken with respect to the independent variable $\eta_1$ given in \eqref{indep2}, with the two critical points $Q_1'=(0,0)$ and $Q_5'=((p-m)/2,0)$ being the restrictions of $Q_1$ and $Q_5$ to the plane $\{x=0\}$. We have the following
\begin{lemma}\label{lem.w0}
The critical point $Q_5'$ is a saddle point in the system \eqref{PSsystw0} and the critical point $Q_1'$ is a saddle-node in the system \eqref{PSsystw0}. The unique orbit contained in the unstable manifold of $Q_5'$ connects to the critical point $Q_3$, with $y(\eta_1)$ decreasing for $\eta_1\in\real$. The orbits going out of $Q_1'$ connect all of them to the critical point $Q_3$ as well, with the function $y(\eta_1)$ having at most one maximum point along these orbits.
\end{lemma}
\begin{proof}
The matrix of the linearization of the system \eqref{PSsystw0} in a neighborhood of $Q_5'$ is
$$
M(Q_5')=\left(
          \begin{array}{cc}
            -\frac{p-m}{2} & -1 \\[1mm]
            0 & \frac{(m+p-2)(p-m)}{2} \\
          \end{array}
        \right),
$$
thus $Q_5'$ is a saddle point. Consider now the isocline where $\dot{y}=0$, that is,
\begin{equation}\label{iso2}
-y^2+\frac{p-m}{2}y-w=0, \qquad {\rm with \ normal} \qquad \overline{n}=\left(-2y+\frac{p-m}{2},-1\right),
\end{equation}
which lies in the region $w\geq0$ if and only if $0\leq y\leq (p-m)/2$. The flow of the system \eqref{PSsystw0} across the isocline is given by the sign of the expression $-(m+p-2)yw<0$, hence it points out into the increasing direction of $w$. This means that an orbit entering the region $\dot{y}<0$ remains there forever. Observing that, in a neighborhood of $Q_5'$, the isocline has the slope
\begin{equation}\label{interm15bis}
w\sim-\frac{p-m}{2}\overline{y}, \qquad \overline{y}=y-\frac{p-m}{2},
\end{equation}
while the slope of the actual orbit contained in the unstable manifold of the point $Q_5'$ is given by the eigenvector corresponding to the second eigenvalue of $M(Q_5')$, that is,
\begin{equation}\label{interm16bis}
w\sim-\frac{(m+p-1)(p-m)}{2}\overline{y},
\end{equation}
we deduce from \eqref{interm15bis}, \eqref{interm16bis} and the fact that $m+p-1>1$ that the orbit stemming from $Q_5'$ enters the region $\dot{y}<0$ (superior to the isocline \eqref{iso2}) and stays forever inside it. This proves that $y(\eta_1)$ is decreasing for $\eta_1\in\real$ along this orbit. It remains to show that this orbit crosses the axis $\{y=0\}$ at a finite point, that is, there is $\eta_1\in\real$ such that $y(\eta_1)=0$. Assume for contradiction that this is not the case, hence $y(\eta_1)>0$ for any $\eta_1\in\real$. We notice from the system \eqref{PSsystw0} that $\dot{y}(\eta_1)<0$ and $\dot{w}(\eta_1)>0$ for any $\eta_1\in\real$, so that, there exist
$$
y_0:=\lim\limits_{\eta_1\to\infty}y(\eta_1)\in\left[0,\frac{p-m}{2}\right], \qquad w_0:=\lim\limits_{\eta_1\to\infty}w(\eta_1)\in(0,\infty].
$$
If $w_0$ is finite, then $(y_0,w_0)$ would be a finite critical point of the system \eqref{PSsystw0} with $y\geq0$, $w>0$ and there is no such point. Thus, $w_0=\infty$ and thus the orbit will form a vertical asymptote from the right as $y\to y_0$. Moreover, since $y(\eta_1)$ is monotone (hence one-to-one) along the orbit stemming from $Q_5'$, for $\eta_1\in\real$, we can apply the inverse function theorem and express $w$ as a function $w(y)$, with derivative
$$
\frac{dw}{dy}=\frac{(m+p-2)yw(y)}{-y^2+(p-m)y/2-w(y)}.
$$
Passing to the limit as $y\to y_0$ and taking into account that $w(y)\to+\infty$, we find that
$$
\lim\limits_{y\to y_0, y>y_0}w'(y)=-(m+p-2)y_0,
$$
which is finite, contradicting the vertical asymptote from the right of the function $w(y)$ at $y=y_0$. We thus infer that there exists $\overline{\eta}_1\in\real$ such that $y(\overline{\eta}_1)=0$ on the orbit contained in the unstable manifold of $Q_5'$, and then $y(\eta_1)<0$ for any $\eta_1>\overline{\eta}_1$. Furthermore, the latter, together with the second equation of the system \eqref{PSsystw0}, implies that $\dot{w}(\eta_1)<0$ for any $\eta_1>\overline{\eta}_1$. This shows that the orbit contained in the unstable manifold of $Q_5'$ has a critical point as limit as $\eta_1\to\infty$, lying in the half-plane $\{y<0\}$ and with finite $w$-coordinate. There are no such finite critical points, and it is then easy to derive that
$$
\lim\limits_{\eta_1\to\infty}y(\eta_1)=-\infty, \qquad \lim\limits_{\eta_1\to\infty}\frac{w(\eta_1)}{y(\eta_1)}=0,
$$
thus the orbit will reach the critical point $Q_3$.

\medskip

We are left with the critical point $Q_1'$. The matrix of the linearization of the system \eqref{PSsystw0} in a neighborhood of $Q_1'$ is given by
$$
M(Q_1')=\left(
          \begin{array}{cc}
            \frac{p-m}{2} & -1 \\[1mm]
            0 & 0 \\
          \end{array}
        \right),
$$
thus we have a one-dimensional unstable manifold (contained in the invariant $y$-axis) and one-dimensional center manifolds. Applying the theory in \cite[Sections 2.4 and 2.5]{Carr} in order to analyze these center manifolds, we find that the Taylor expansion of them up to the second order is given by
$$
w=\frac{p-m}{2}y-(m+p-1)y^2+o(y^2),
$$
while the flow on these center manifolds in a neighborhood of the origin is given by
$$
\dot{w}=\frac{(m+p-2)(p-m)}{2}y^2+o(y^2)>0.
$$
Thus, all the center manifolds have an unstable flow, all these trajectories going out into the positive half-plane $\{y>0\}$ tangent to the line
$$
w=\frac{p-m}{2}y.
$$
They have to enter the closed region limited by the axis $\{w=0\}$ (which is invariant) and the isocline \eqref{iso2}, that is, the region where $\dot{y}>0$. Since the flow on the isocline \eqref{iso2} points into an unique direction (towards the open, exterior region where $\dot{y}<0$), it follows that all the orbits contained in center manifolds going out of $Q_1'$ will cross the isocline \eqref{iso2} at some point (and they can do this only once, due to the direction of the flow across \eqref{iso2}) and then enter and stay forever in the region where $\dot{y}<0$. Since they are bounded from above by the unique orbit contained in the unstable manifold of $Q_5'$, they have then to cross the axis $\{y=0\}$ at some finite point and enter the half-plane $\{y<0\}$. The same considerations as in the final part of the proof related to the orbit contained in the unstable manifold of $Q_5'$ show that all these orbits will end up connecting to $Q_3$, and $y(\eta_1)$ has a unique maximum point, which is exactly at the value of $\eta_1$ for which the orbit intersects the isocline \eqref{iso2}.
\end{proof}

\section{Proof of Theorem \ref{th.1}}\label{sec.main}

In this section we prove Theorem \ref{th.1}. Let then fix $p>p_s$ throughout all this section. We begin with several preparatory results on the global analysis of the system \eqref{PSsyst}.
\begin{lemma}\label{lem.tang}
Let $(X(\eta),Y(\eta),Z(\eta))$ with $\eta\in\real$ be a trajectory of the system \eqref{PSsyst}. If there is $\eta_0\in\real$ such that $Y(\eta_0)=\dot{Y}(\eta_0)=0$, then the trajectory is the explicit one
\begin{equation}\label{const.orbit}
X(\eta)=Z(\eta), \quad Y(\eta)=0, \quad {\rm for \ any} \ \eta\in\real.
\end{equation}
If in change, there is $\eta_1\in\real$ such that $Y(\eta_1)=-2/(p-m)$, $\dot{Y}(\eta_1)=0$, then the trajectory is the explicit one
\begin{equation}\label{stat.orbit}
Z(\eta)=Z_0:=\frac{c_s^{p-m}}{m}\frac{2[p(N-2)-mN]}{(p-m)^2}, \quad Y(\eta)=-\frac{2}{p-m}, \quad {\rm for \ any} \ \eta\in\real.
\end{equation}
\end{lemma}
\begin{proof}
In the former case, we deduce from the fact that $Y(\eta_0)=\dot{Y}(\eta_0)=0$ and the second equation of the system \eqref{PSsyst} that $X(\eta_0)=Z(\eta_0)$. Since \eqref{const.orbit} is a trajectory of the system \eqref{PSsyst}, the uniqueness of a trajectory passing through the point $(X(\eta_0),0,X(\eta_0))$ (which is not critical for \eqref{PSsyst}) implies the coincidence with \eqref{const.orbit}. In the latter case, we deduce from the fact that $Y(\eta_1)=-2/(p-m)$, $\dot{Y}(\eta_1)=0$ and the second equation of the system \eqref{PSsyst} that 
$$
Z(\eta_1)=\frac{2[p(N-2)-mN]}{(p-m)^2},
$$
and since \eqref{stat.orbit} is a trajectory of the system \eqref{PSsyst}, a similar uniqueness argument completes the proof.
\end{proof}

\noindent \textbf{Remark.} Notice that the orbit \eqref{const.orbit} corresponds to the constant profile $f(\xi)=k_*$ of the solution $U_*$ given in \eqref{const.sol}, while the orbit \eqref{stat.orbit} corresponds to the profile $U_s$ of the stationary solution given in \eqref{stat.sol}. 

\begin{lemma}\label{lem.maxmin}
Let $(X(\eta),Y(\eta),Z(\eta))$ with $\eta\in\real$ be a trajectory of the system \eqref{PSsyst}. Then, the function $\eta\mapsto Y(\eta)$ cannot have a local maximum in the region
\begin{equation}\label{reg1}
\mathcal{R}_1=\left\{X\geq\frac{2(N-2)}{p-1}, -\frac{2}{p-m}<Y<0\right\},
\end{equation}
and it cannot have a local minimum in the region
\begin{equation}\label{reg2}
\mathcal{R}_2=\left\{X\geq\frac{2(N-2)}{p-1}, -\infty<Y<-\frac{2}{p-m}\right\}.
\end{equation}
\end{lemma}
\begin{proof}
Assume for contradiction that there is $\eta_0\in\real$ such that $Y(\eta_0)$ is a local maximum for the function $\eta\mapsto Y(\eta)$ with $(X(\eta_0),Y(\eta_0),Z(\eta_0))\in\mathcal{R}_1$. We thus have $-2/(p-m)<Y(\eta_0)<0$, $\dot{Y}(\eta_0)=0$ and $Y''(\eta_0)\leq0$. We next differentiate the second equation in \eqref{PSsyst} with respect to $\eta$ to get
\begin{equation*}
\begin{split}
Y''(\eta_0)&=\dot{X}(\eta_0)\left[1+\frac{p-m}{2}Y(\eta_0)\right]-\dot{Z}(\eta_0)\\
&=X(\eta_0)[2-(m-1)Y(\eta_0)]\left[1+\frac{p-m}{2}Y(\eta_0)\right]-Z(\eta_0)[2+(p-m)Y(\eta_0)]\\
&=[2+(p-m)Y(\eta_0)]\left[X(\eta_0)\left(1-\frac{m-1}{2}Y(\eta_0)\right)-Z(\eta_0)\right]
\end{split}
\end{equation*}
Since $\dot{Y}(\eta_0)=0$, we can replace $Z(\eta_0)$ from the second equation of the system \eqref{PSsyst} and continue the previous calculation as follows:
\begin{equation}\label{interm21}
\begin{split}
Y''(\eta_0)&=[2+(p-m)Y(\eta_0)]\left[X(\eta_0)-\frac{m-1}{2}X(\eta_0)Y(\eta_0)+(N-2)Y(\eta_0)\right.\\&\left.+mY^2(\eta_0)-X(\eta_0)-\frac{p-m}{2}X(\eta_0)Y(\eta_0)\right]\\
&=[2+(p-m)Y(\eta_0)]Y(\eta_0)\left[-\frac{p-1}{2}X(\eta_0)+N-2+mY(\eta_0)\right].
\end{split}
\end{equation}
Since $-2/(p-m)<Y(\eta_0)<0$, the third factor in the expression of $Y''(\eta_0)$ in \eqref{interm21} must be non-negative in order to ensure that $Y''(\eta_0)\leq0$. But, taking into account that $Y(\eta_0)<0$, this is only possible if
$$
X(\eta_0)<\frac{2(N-2)}{p-1},
$$
in contradiction with the definition of the region $\mathcal{R}_1$ given in \eqref{reg1}. This contradiction shows that there are no local maxima in the region $\mathcal{R}_1$. A completely similar argument with a point of minimum $Y(\eta_1)$ in the region $\mathcal{R}_2$ gives that once more the third term in \eqref{interm21} must be non-negative and we reach exactly the same contradiction.
\end{proof}
Our next two preparatory lemmas concern the global behavior of trajectories of the system \eqref{PSsyst} in the regions of the $(X,Y,Z)$ space limited by some important planes.
\begin{lemma}\label{lem.globalQ3}
Let $(X(\eta),Y(\eta),Z(\eta))$ with $\eta\in\real$ be a trajectory of the system \eqref{PSsyst}. If there exists $\eta_0\in\real$ such that
$$
Y(\eta_0)<-\frac{N-2}{m},
$$
then the trajectory enters the stable node $Q_3$. 
\end{lemma}
\begin{proof}
Let us consider the plane $\{Y=-(N-2)/m\}$. The direction of the flow of the system \eqref{PSsyst} across this plane (with normal direction $(0,1,0)$) is given by the sign of the expression
$$
E(X,Z)=\frac{mN-p(N-2)}{2m}X-Z<0,
$$
since $p>p_s$. We thus conclude that $Y(\eta)<-(N-2)/m$ for any $\eta\in(\eta_0,\infty)$. Noticing also that
$$
\frac{2}{p-m}-\frac{N-2}{m}=\frac{mN-p(N-2)}{m(p-m)}<0,
$$
for any $p>p_s$, we infer from the first and third equations of the system \eqref{PSsyst} that $\dot{X}(\eta)>0$ and $\dot{Z}(\eta)<0$ for any $\eta\in(\eta_0,\infty)$. We also have
\begin{equation}\label{interm18}
\dot{Y}(\eta)=-Y(\eta)(mY(\eta)+N-2)+X(\eta)\left(1+\frac{p-m}{2}Y(\eta)\right)-Z(\eta)<0,
\end{equation}
provided $Y(\eta)<-(N-2)/m$, whence all the coordinates $(X(\eta),Y(\eta),Z(\eta))$ are monotone for $\eta\in(\eta_0,\infty)$ and thus converge to limits $(X_{\infty},Y_{\infty},Z_{\infty})$ as $\eta\to\infty$, where $0\leq Z_{\infty}\leq Z(\eta_0)<\infty$. Observe that there is no finite critical point with these properties, so that either $Y_{\infty}=-\infty$, or $X_{\infty}=+\infty$. 

Assume for contradiction that $Y_{\infty}>-\infty$, and thus necessarily $X_{\infty}=+\infty$. In this case, we readily obtain from \eqref{interm18} that $\dot{Y}(\eta)\to\infty$ as $\eta\to\infty$, which is a contradiction with the fact that $Y_{\infty}$ is finite, as it can be seen from \cite[Lemma 2.9]{IL13} stating that there should be a subsequence $(\eta_j)_{j\geq1}$ such that $\eta_j\to\infty$ and $\dot{Y}(\eta_j)\to0$ as $j\to\infty$. It thus follows that $Y_{\infty}=-\infty$. If $X_{\infty}<\infty$, then the critical point $(X_{\infty},-\infty,Z_{\infty})$ is identified as $Q_3$ on the Poincar\'e hypersphere, since $Y(\eta)\to-\infty$ and dominates over the other two components. In the case $X_{\infty}=\infty$, we infer from the monotonicity of $X(\eta)$ and the inverse function theorem that, for $\eta\in(\eta_0,\infty)$, we can invert the dependence $X(\eta)$ and thus define a function $Y=Y(X)$ along the given trajectory, which, for $X$ very large, satisfies (keeping only the dominating terms)
$$
\frac{dY}{dX}\sim\frac{(p-m)XY/2-mY^2}{-(m-1)XY}=\frac{2mY-(p-m)X}{2(m-1)X}.
$$
By integration, the previous equivalence leads to
$$
Y(\eta)\sim CX^{m/(m-1)}(\eta)+\frac{p-m}{2}X(\eta), \qquad C<0, \qquad {\rm as} \ \eta\to\infty.
$$
Since $m>m-1$, we still infer that $Y(\eta)/X(\eta)\to-\infty$ as $\eta\to\infty$, so that our given trajectory connects to the node $Q_3$.
\end{proof}
\begin{lemma}\label{lem.globalQ1}
Let $(X(\eta),Y(\eta),Z(\eta))$ with $\eta\in\real$ be a trajectory of the system \eqref{PSsyst}. If there is $\eta_*\in\real$ such that 
$$
-\frac{N-2}{m}<Y(\eta)<0, \qquad {\rm for \ any} \ \eta\in(\eta_*,\infty),
$$
then the trajectory connects to the critical point $Q_1$.
\end{lemma}
\begin{proof}
We work on the interval $(\eta_*,\infty)$. Since $Y(\eta)<0$ for any $\eta>\eta_*$, we infer from the first equation of the system \eqref{PSsyst} that $\dot{X}(\eta)>2X(\eta)>0$ for any $\eta>\eta_*$, which implies that $X(\eta)$ is increasing on $(\eta_*,\infty)$ and $X(\eta)\to\infty$ as $\eta\to\infty$. By eventually increasing $\eta_*$ such that we also have $X(\eta_*)>2(N-2)/(p-1)$, it follows from Lemma \ref{lem.maxmin} that there is no local maximum for $Y(\eta)$ such that $-2/(p-m)<Y(\eta)<0$ for any $\eta>\eta_*$. We have thus two possible situations:

$\bullet$ either $-2/(p-m)<Y(\eta)<0$ for any $\eta\in(\eta_*,\infty)$. In this case, according to Lemma \ref{lem.maxmin}, the mapping $\eta\mapsto Y(\eta)$ is either decreasing on $(\eta_*,\infty)$, or it has a minimum $\eta^*\in(\eta_*,\infty)$ and then it stays increasing on $(\eta^*,\infty)$, while the third equation in the system \eqref{PSsyst} ensures that $Z$ is increasing on $(\eta_*,\infty)$. We thus infer that there are limits $Y_{\infty}\in[-2/(p-m),0]$ and $Z_{\infty}>0$ of $Y(\eta)$, respectively $Z(\eta)$ as $\eta\to\infty$. If the limit of $Y(\eta)$ as $\eta\to\infty$ is negative, then there are $\epsilon>0$ and $\eta_{\epsilon}\in\real$ such that $Y(\eta)<-\epsilon$ for any $\eta\in(\eta_{\epsilon},\infty)$. Since $X$ is a monotone function of $\eta\in(\eta_{\epsilon},\infty)$ on the given trajectory, we can invert the dependence $X(\eta)$ on the interval $(\eta_{\epsilon},\infty)$ to get a well-defined function $\eta=\eta(X)$ and compose this inverse with $Z(\eta)$ to obtain a well-defined dependence $Z=Z(X)$ along the given trajectory. The inverse function theorem then easily entails that
\begin{equation}\label{interm19}
Z'(X)=\frac{Z(2+(p-m)Y)}{X(2-(m-1)Y)}<\frac{Z}{X}\frac{2-(p-m)\epsilon}{2+(m-1)\epsilon}=K_{\epsilon}\frac{Z}{X},
\end{equation}
with $K_{\epsilon}<1$. By comparing with the solutions of the differential equation given by exact equality in \eqref{interm19} and then undoing the inversion with respect to $\eta$, we find that $Z(\eta)<CX(\eta)^{K_{\epsilon}}$ for some $C>0$, for any $\eta\in(\eta_{\epsilon},\infty)$. In particular, and recalling also that $-(N-2)/m<Y(\eta)<0$ for any $\eta\in\real$, we arrive to the conclusion that
\begin{equation}\label{interm20}
\lim\limits_{\eta\to\infty}\frac{Z(\eta)}{X(\eta)}=0=\lim\limits_{\eta\to\infty}\frac{Y(\eta)}{X(\eta)},
\end{equation}
which, together with the fact that $X(\eta)\to\infty$ as $\eta\to\infty$, shows that the trajectory $l_C$ enters the critical point $Q_1$. In the case when $Y(\eta)\to0$ as $\eta\to\infty$, we can proceed as in \eqref{interm19} but only to deduce that $Z(\eta)\leq CX(\eta)$ for some $C>0$ and as $\eta\to\infty$. It then follows that the given trajectory either enters $Q_1$ (if $Z(\eta)/X(\eta)\to0$ as $\eta\to\infty$) or approaches some critical points $Q_{\gamma}$. But the latter is not possible, as it contradicts the outcome of Lemma \ref{lem.Qg}.

$\bullet$ or there exists $\tilde{\eta}\in(\eta_*,\infty)$ such that $Y(\tilde{\eta})=-2/(p-m)$ and $X(\tilde{\eta})>2(N-2)/(p-1)$. We infer from Lemmas \ref{lem.tang} and \ref{lem.maxmin} that
$$
-\frac{N-2}{m}<Y(\eta)\leq-\frac{2}{p-m}, \qquad \eta\in(\tilde{\eta},\infty),
$$
whence by the third equation in the system \eqref{PSsyst}, $Z$ is non-increasing in the interval $(\tilde{\eta},\infty)$ and thus there exists
$$
Z_{\infty}=\lim\limits_{\eta\to\infty}Z(\eta)\in[0,Z(\tilde{\eta})).
$$
Since $X(\eta)\to\infty$ as $\eta\to\infty$, we deduce again that \eqref{interm20} holds true and thus our given trajectory connects to the critical point $Q_1$.
\end{proof}
We end these preparations by a result establishing a lower bound for the trajectories $(l_C)_{C\in[1,\infty]}$.
\begin{lemma}\label{lem.lower}
The trajectories $(l_C)_{C\in[1,\infty]}$ on the unstable manifold of $P_0$ lie above the parabolic cylinder 
\begin{equation}\label{cyl2}
\mathcal{H}=\left\{(X,Y,Z)\in\real^3: Z=-(N-2)Y-mY^2, \ X\geq0, \ -\frac{2}{p-m}\leq Y\leq0\right\}
\end{equation}
\end{lemma}
\begin{proof}
It is easy to see from Lemma \ref{lem.P0} that the orbits lying on the unstable manifold of $P_0$ have the local behavior 
$$
Y(\eta)\sim\frac{X(\eta)-Z(\eta)}{N}, \qquad {\rm as} \ \eta\to-\infty,
$$
or equivalently,
$$
Z(\eta)\sim X(\eta)-NY(\eta)\geq-NY(\eta)\geq-(N-2)Y(\eta)-mY(\eta)^2, \qquad {\rm as} \ \eta\to-\infty,
$$
hence they go out in a neighborhood of $P_0$ above the surface $\mathcal{H}$ for $C\in[1,\infty]$, due to the non-positivity of $Y(\eta)$. Moreover, the direction of the flow of the system \eqref{PSsyst} across the surface \eqref{cyl2}, with normal direction $\overline{n}=(0,(N-2)+2mY,1)$, is given by the sign of the expression 
$$
H(X,Y)=\frac{1}{2}(2+(p-m)Y)\big[X(2mY+N-2)-2Y(mY+N-2)\big].
$$
Notice that, on the surface $\mathcal{H}$ defined in \eqref{cyl2}, we readily have $2+(p-m)Y\geq0$, $-2Y(mY+N-2)\geq0$ and 
$$
2mY+N-2\geq-\frac{4m}{p-m}+N-2=\frac{(N-2)(p-p_s)}{p-m}>0,
$$
since we are working with $p>p_s$. Thus $H(X,Y)\geq0$ on $\mathcal{H}$ and we conclude that the orbits $(l_C)_{C\in[1,\infty]}$ cannot cross the surface $\mathcal{H}$ before, at least, crossing the plane $\{Y=-2/(p-m)\}$, as claimed.
\end{proof}
We now move towards the proof of Theorem \ref{th.1}. In order to fix the terminology, we say that a trajectory $(X(\eta),Y(\eta),Z(\eta))$ of the system \eqref{PSsyst} \textbf{has an oscillation} with respect to the plane $\{Y=0\}$ if there exist $\eta_1<\eta_2\in\real$ such that $Y(\eta_1)=Y(\eta_2)=0$ and $Y(\eta)>0$ for $\eta\in(\eta_1,\eta_2)$. This is well-defined according to Lemma \ref{lem.tang} and extends in an obvious way to the definition of $K$ oscillations with respect to the plane $\{Y=0\}$ for any natural number $K\geq2$. 

Let us recall next how the Lepin exponent $p_L$ comes into play, according to \cite[Theorem 12.2]{GV97}. Starting from \eqref{SSODE}, we introduce the function $g(\xi)$ by
\begin{equation}\label{change0}
f^m(\xi)=\xi^{-2m/(p-m)}g(\xi)
\end{equation}
and recalling the independent variable $\eta=\ln\,\xi$ already employed in the system \eqref{PSsyst}, we deduce the differential equation solved by $g(\eta)$
\begin{equation}\label{SSODE2}
g''(\eta)+Ag'(\eta)-\frac{mB}{p-m}g(\eta)+g^{p/m}(\eta)-\frac{\beta}{m}(g^{(1-m)/m}g')(\eta)e^{\eta/\beta},
\end{equation}
where $\beta=(p-m)/2(p-1)$ and the coefficients are
\begin{equation}\label{coefs}
A=N-2-\frac{4m}{p-m}, \qquad B=2\left(N-2-\frac{2m}{p-m}\right).
\end{equation}
By the transformation \eqref{change0}, the stationary solution $U_s$ given in \eqref{stat.sol} is mapped into the constant function $G_s(\eta)=c_s^m$. We then linearize the equation \eqref{SSODE2} in a neighborhood of this constant solution by letting 
\begin{equation}\label{interm22}
g(\eta)=c_s^m+h(\eta),
\end{equation}
and in a first approximation one obtains the linear equation in general form 
\begin{equation}\label{SSODE3}
h''(\eta)+Ah'(\eta)+Bh(\eta)-Ce^{\eta/\beta}h'(\eta)=0,
\end{equation}
for $A$, $B$ given in \eqref{coefs} and some $C>0$. All the previous calculations are given in the proof of \cite[Theorem 12.2]{GV97}. Linear equations in the general form \eqref{SSODE3} have been studied by Lepin in \cite{Le89, Le90}, and it is established that for $p<p_L$ the function $h$ solution to \eqref{SSODE3} with the condition $h(\eta)\to C$ as $\eta\to\infty$ has at least 3 zeros. This is the main argument employed in the existence proof for $p_s<p<p_L$ in \cite{GV97}. However, \cite[Lemma 9]{Le90} also establishes that, if $p\geq p_L$, any solution to \eqref{SSODE3} such that $h(\eta)\to1$, and therefore also with the property $h(\eta)\to C$ for any $C>0$ (since \eqref{SSODE3} is a linear equation) as $\eta\to\infty$, has exactly two zeros. We can next state the following technical result translating the outcome of \cite[Lemma 9]{Le90} into trajectories of the system \eqref{PSsyst}.
\begin{lemma}\label{lem.crossP2}
There exists $K_0>0$ sufficiently small such that any trajectory of the system \eqref{PSsyst} entering the critical point $Q_1$ with
$$
\lim\limits_{\eta\to\infty}Z(\eta)=k\in(c_s^{p-m},c_s^{p-m}+K_0),
$$
(according to Lemma \ref{lem.Q1} and the remark at its end) has at least one intersection with the plane $\{Y=-2/(p-m)\}$
\end{lemma}
\begin{proof}
We translate in our notation and terminology what the existence of exactly two zeros of solutions $h(\eta)$ to \eqref{SSODE3} means. Observe first that, in terms of profiles, any zero of the function $h(\eta)$ means an intersection of the profile $f(\xi)$ with the stationary solution $U_s$ given in \eqref{stat.sol}. Moreover, we also remark that there is a one-to-one mapping between $(f,\xi)$ and $(X,Z)$ variables, according to \eqref{PSchange}. Indeed, given $(X,Z)$ as in \eqref{PSchange}, we find that
$$
f=\left[\frac{Z}{(p-1)X}\right]^{1/(p-1)}, \qquad \xi=\sqrt{m}[(p-1)X]^{\beta}Z^{(m-1)/2(p-1)}.
$$
Thus, an intersection of a solution to \eqref{SSODE} with the profile $U_s$ means the same values of $(X,Z)$ in the phase space of \eqref{PSsyst}. Since $U_s$ identifies with the orbit \eqref{stat.orbit}, we conclude that a zero of the linearized function $h(\eta)$ solution to \eqref{SSODE3} is seen in the space $(X,Y,Z)$ of the system \eqref{PSsyst} as a point such that $X>0$, $Z=Z_0$, where $Z_0$ is defined in \eqref{stat.orbit}. Thus, \cite[Lemma 9]{Le90} implies that the corresponding trajectories will intersect twice the half-plane $\{Z=Z_0, X>0\}$. Since the third equation of the system \eqref{PSsyst} implies that $Z(\eta)$ is increasing while $Y(\eta)>-2/(p-m)$ and decreasing while $Y(\eta)<-2/(p-m)$, the existence of two intersections with $\{Z=Z_0, X>0\}$ is only possible if the trajectory intersects the plane $\{Y=-2/(p-m)\}$ at least once. Finally, since in the above we work with a linear approximation of the true equation for $h(\eta)$, this stays true for trajectories lying sufficiently close to the constant $c_s^m$ in \eqref{interm22}, that is, for trajectories such that $Z(\eta)\to k\in(Z_0,Z_0+K_0)$ for some $K_0>0$ sufficiently small, as claimed.
\end{proof}

With these are now ready to complete the proof of Theorem \ref{th.1} for $p>p_s$.
\begin{proof}[Proof of Theorem \ref{th.1}]
In view of Lemmas \ref{lem.P0} and \ref{lem.Q1}, we are looking for trajectories of the system \eqref{PSsyst} connecting the critical points $P_0$ and $Q_1$. We perform thus a shooting on the unstable manifold of the critical point $P_0$, recalling that the trajectories contained in its unstable manifold form a one-parameter family $(l_C)_{C\geq0}$, according to Lemma \ref{lem.P0}. Notice that $l_1$ is the trajectory \eqref{const.orbit}, and let us denote by $l_{\infty}$ the trajectory contained in the invariant plane $\{X=0\}$, which enters the critical point $P_2$ according to Lemma \ref{lem.X0}. We moreover recall from \eqref{interm2} that the orbits $l_C$ with $C\in(1,\infty)$ all go out into the negative half-space $\{Y<0\}$, that is, on each of them $Y(\eta)<0$ in a neighborhood of $-\infty$.

\medskip 

\noindent \textbf{Proof of Part (a)}. We split the interval $(1,\infty)$ into the following three sets: 
\begin{equation}\label{sets}
\begin{split}
&\mathcal{A}=\left\{C\in(1,\infty): {\rm there \ is} \ \eta_0\in\real, \ -\frac{2}{p-m}<Y(\eta)<0 \ {\rm for \ any} \ \eta\in(-\infty,\eta_0),\right.\\&\left. {\rm and} \ Y(\eta_0)=-\frac{2}{p-m}, \ \dot{Y}(\eta_0)<0\right\},\\
&\mathcal{C}=\left\{C\in(1,\infty): {\rm there \ is} \ \eta_0\in\real, \ -\frac{2}{p-m}<Y(\eta)<0 \ {\rm for \ any} \ \eta\in(-\infty,\eta_0),\right.\\&\left. {\rm and} \ Y(\eta_0)=0, \ \dot{Y}(\eta_0)>0\right\},\\
&\mathcal{B}=(1,\infty)\setminus(\mathcal{A}\cup\mathcal{C}).
\end{split}
\end{equation}
The definition readily imply that $\mathcal{A}$ and $\mathcal{C}$ are open sets. We show next that they are non-empty. 

The most involved part of the proof is to show that $\mathcal{A}$ is non-empty. To this end, we employ the following strategy: assuming for contradiction that $\mathcal{A}=\emptyset$, we build a system of joined tubular neighborhoods of the orbits (in opposite direction to the flow) \eqref{stat.orbit} connecting $P_2$ to $Q_1$, the orbit $l_{\infty}$ connecting $P_0$ to $P_2$ and the orbit contained in the $Y$ axis in the region $\{Y>0\}$ connecting $Q_2$ to $P_0$. We take the neighborhoods as thin as needed in order for a continuum of orbits $(l_C)_{C>C^*}$ for some $C^*>1$ to form a barrier for the trajectories inside the first two of these three tubular neighborhoods. With this construction, we will conclude that there is at least one trajectory connecting $Q_2$ to $Q_1$ for which $Y(\eta)>-2/(p-m)$ for any $\eta\in\real$, contradicting the outcome of Lemma \ref{lem.crossP2}. We next give the details of this construction.

Assume for contradiction, on the one hand, that $\mathcal{A}=\emptyset$ and also $\mathcal{B}=\emptyset$ (otherwise by Lemma \ref{lem.globalQ1} we would already have a solution and the proof would be completed), that is, for any $C\in(1,\infty)$, the trajectory $l_C$ meets first the plane $\{Y=0\}$ before intersecting the plane $\{Y=-2/(p-m)\}$. Then, on the orbit $l_{\infty}$ contained in the plane $\{X=0\}$ we must have $0\leq Y(\eta)<-2/(p-m)$ for any $\eta\in\real$, as otherwise continuity would imply that some $l_C$ with $C$ very large also crosses $\{Y=-2/(p-m)\}$. Pick $\epsilon\in(0,K_0)$ (with $K_0$ chosen so that Lemma \ref{lem.crossP2} holds true) such that \cite[Theorem 5.11]{Shilnikov} describing the local behavior of trajectories near a (generalized) saddle holds true in the neighborhood
\begin{equation}\label{neigh.Q1}
\mathcal{U}=\{0<x<\epsilon, |y|<\epsilon,0<w=xz<\epsilon\}
\end{equation}
of the critical point $Q_1$ in the system \eqref{PSinf1}. Taking into account that $w=xz$ and the change of variable \eqref{change2}, we readily infer that the neighborhood
\begin{equation}\label{neigh.Q1.main}
\begin{split}
\mathcal{V}=&\left\{(X,Y,Z)\in\real^3:X>\max\left\{\frac{1}{\epsilon},\frac{2}{(p-m)\epsilon},\sqrt{\frac{Z_0+\epsilon}{\epsilon}}\right\},\right.\\&\left. -\frac{2}{p-m}<Y<-\frac{2}{p-m}+\epsilon, Z_0-\epsilon<Z<Z_0+\epsilon\right\},
\end{split}
\end{equation}
where $Z_0$ is defined in \eqref{stat.orbit}, is contained in $\mathcal{U}$. Since by Lemma \ref{lem.X0}, the orbit $l_{\infty}$ enters the critical point $P_2$, the local behavior near a hyperbolic saddle (see \cite[Theorem 2.9, Section 2.8]{Shilnikov}), the fact that the only orbit in the unstable manifold of $P_2$ is the line \eqref{stat.orbit} and the assumption that the orbits $l_C$ do not cross the plane $\{Y=-2/(p-m)\}$ entail that there exists $C^*>1$ such that the orbit $l_C$ enters the neighborhood $\mathcal{V}$ defined in \eqref{neigh.Q1.main} for any $C\in(C^*,\infty)$, and then they approach the unstable sector of $Q_1$ analyzed in Lemma \ref{lem.w0} forcing them to intersect first the plane $\{Y=0\}$. We then construct a "tubular" right-neighborhood of the straight line \eqref{stat.orbit} which extends $\mathcal{V}$, namely
\begin{equation*}
\begin{split}
\mathcal{W}=&\left\{(X,Y,Z)\in\real^3: 0\leq X\leq\max\left\{\frac{1}{\epsilon},\frac{2}{(p-m)\epsilon},\sqrt{\frac{Z_0+\epsilon}{\epsilon}}\right\},\right.\\&\left.-\frac{2}{p-m}<Y<-\frac{2}{p-m}+\epsilon, Z_0-\epsilon<Z<Z_0+\epsilon\right\}
\end{split}
\end{equation*}
We complete this tubular neighborhood with a tubular half-neighborhood of the orbit $l_{\infty}$ defined as follows
\begin{equation*}
\mathcal{N}=\left\{(X,Y,Z)\in\real^3: 0\leq X<\epsilon, -(N-2)Y-mY^2<Z<Z_0+\epsilon, -\frac{2}{p-m}\leq Y\leq0\right\}.
\end{equation*}
Moreover, Lemma \ref{lem.lower} ensures that the orbits $l_C$ on the unstable manifold of $P_0$ for $C>1$ sufficiently large will enter and remain in the neighborhood $\mathcal{N}$ up to the region $\mathcal{N}\cap\mathcal{W}$.

\medskip

In the next step, we select our $\epsilon>0$. The assumption by contradiction, together with an application of \cite[Theorem 2.9]{Shilnikov} in a neighborhood of $P_2$ and an application of \cite[Theorem 5.11]{Shilnikov} in a neighborhood of $Q_1$ entail that there is $C^*\in(1,\infty)$ such that the orbits $(l_C)_{C\in[C^*,\infty]}$ form a continuum both in a (fixed) neighborhood $\mathcal{\overline{U}}$ of $P_2$ and a (fixed) neighborhood $\mathcal{\overline{V}}$ of $Q_1$. Lemma \ref{lem.tang} then gives that the distance $d_{C^*}$ between the part of the trajectory $(l_{C^*})$ lying between the two neighborhoods and the plane $\{Y=-2/(p-m)\}$ is positive. Choose $\epsilon<\min\{d_{C^*},K_0\}$. Letting now any plane $\{X=X_0\}$ for $\epsilon<X_0<1/\epsilon$, its intersection with the unstable manifold of $P_0$ describes a continuous curve with one endpoint on the trajectory \eqref{stat.orbit} and the other endpoint one the curve $(l_{C^*})$. Since $\epsilon<d_{C^*}$, we derive that the unstable manifold stemming from $P_0$ splits the ``tubular" neighborhood $\mathcal{W}\cup\mathcal{V}$ into two disjoint parts. A similar argument (eventually choosing $\epsilon$ even smaller) allows to show that the unstable manifold from $P_0$ splits also into two disjoint parts the set $\mathcal{N}$.

\medskip 

Let us shoot now backwards (in the reversed direction of the flow) from $Q_1$ with orbits (denoted by $r_K$) satisfying
\begin{equation}\label{interm23}
(r_K): \quad \lim\limits_{\eta\to\infty}Z(\eta)=K\in[Z_0,Z_0+\epsilon)\subset(Z_0,Z_0+K_0),
\end{equation}
according to Lemma \ref{lem.Q1} and the Remark at its end. By continuity with respect to the parameter $K$ on the two-dimensional center manifold of $Q_1$, we deduce that, for $K-Z_0$ sufficiently small, the orbits $r_K$ will be as close as we wish to the straight line \eqref{stat.orbit} (which corresponds to \eqref{interm23} with $K=Z_0$). Since the flow on the plane $\{Y=-2/(p-m)\}$, with normal vector $(0,1,0)$, is given by the sign of $Z_0-Z$, these orbits $r_K$ lying inside the neighborhood $\mathcal{V}\cup\mathcal{W}$ cannot cross the boundary formed by the unstable manifold of $P_0$ together with the region $\{Y=-2/(p-m), Z\geq Z_0\}$. Hence, they will remain in the neighborhood $\mathcal{V}\cup\mathcal{W}$ until approaching $P_2$ in the sense that for any such $K$ there is $\eta_K\in\real$ such that $(X,Y,Z)(\eta_K)\in\mathcal{W}\cap\mathcal{N}$. The behavior near the saddle $P_2$ \cite[Theorem 2.9]{Shilnikov} applied for $(X,Y,Z)\in\mathcal{W}\cap\mathcal{N}$ implies that these orbits $r_K$ for $K-Z_0$ sufficiently small come from the neighborhood $\mathcal{N}$. Since they cannot intersect neither the unstable manifold of $P_0$, nor the invariant plane $\{X=0\}$ while in $\mathcal{N}$, it follows that they have to stay in $\mathcal{N}$ until reaching the plane $\{Y=0\}$. A similar argument employing the local behavior near a saddle \cite[Theorem 2.9]{Shilnikov} in a neighborhood of $P_0$ included in $\mathcal{N}$ then shows that these orbits $r_K$ as above come from a "tubular" neighborhood of the unique orbit in the stable manifold of $P_0$ in the half-space $\{Y>0\}$, which is contained in the $Y$ axis according to Lemma \ref{lem.P0} and connects the unstable node $Q_2$ to $P_0$. In conclusion, the sequence of neighborhoods and the previous arguments allow us to find trajectories $r_K$ for $K$ sufficiently close to $Z_0$ stemming from $Q_2$ and not crossing at all the plane $\{Y=-2/(p-m)\}$. But this is a contradiction with Lemma \ref{lem.crossP2} and \eqref{interm23}. This contradiction implies that $\mathcal{A}$ is non-empty.

\medskip

On the other hand, the analysis performed in \cite[pp. 191-193]{S4}, based on the study of the linearization of \eqref{SSODE} with respect to the constant solution $f(\xi)=k_*$ introduced in \eqref{const.sol} and the number of zeros of the resulting Kummer series, proves that there is $C_*>1$ such that for any $C\in(1,C_*)$, the trajectory $l_C$ has at least one oscillation with respect to the plane $\{Y=0\}$. This proves that $\mathcal{C}$ is non-empty and it contains a right neighborhood $(1,C_*)$ of $C=1$. Standard topological arguments give then that the set $\mathcal{B}$ is non-empty. Picking any $C\in\mathcal{B}$, since $C\notin\mathcal{A}$ and $C\notin\mathcal{C}$, we infer from Lemma \ref{lem.tang} that $-2/(p-m)<Y(\eta)<0$ for any $\eta\in\real$ along the trajectory $l_C$. Lemma \ref{lem.globalQ1} then entails that the trajectory $l_C$ connects to the critical point $Q_1$. Observe that we have just obtained, by undoing the transformation \eqref{PSchange}, a self-similar profile which is monotone decreasing. Moreover, we deduce by dividing by $x^2$ in \eqref{cmf} in order to go back to variables $(X,Y,Z)$ that the trajectories $l_C$ with $C\in\mathcal{B}$ enter the center manifold of $Q_1$ with $Z>Z_0$, where $Z_0$ is defined in \eqref{stat.orbit}, hence $f(\xi)\sim C\xi^{-2/(p-m)}$ for some $C>c_s$ as $\xi\to\infty$, as claimed. We have plotted in Figure \ref{fig1} the outcome of some numerical experiments showing various trajectories $l_C$ for $C\in\mathcal{A}$, respectively $C\in\mathcal{C}$, according to \eqref{sets}, illustrating the general idea of the proof. We also show how their corresponding profiles behave in the $(\xi,f)$ variables.

\begin{figure}[ht!]
  \begin{center}
  \subfigure[Various trajectories in the sets $\mathcal{A}$ and $\mathcal{C}$]{\includegraphics[width=7.5cm,height=6cm]{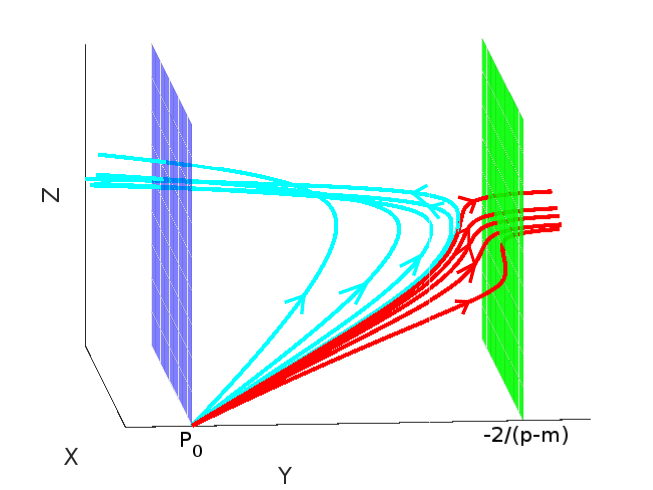}}
  \subfigure[Profiles corresponding to the previous trajectories]{\includegraphics[width=7.5cm,height=6cm]{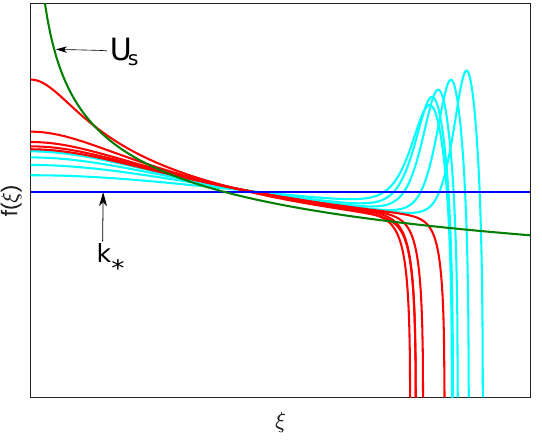}}
  \end{center}
  \caption{Various trajectories in the sets $\mathcal{A}$ and $\mathcal{C}$, seen in the phase space and in profiles. Experiments for $m=2$, $N=20$, $p=10$, with $p_L=3.2$.}\label{fig1}
\end{figure}

Notice also that, in particular, $C_0=\sup\mathcal{C}\in\mathcal{B}$, since $\mathcal{C}$ and $\mathcal{A}$ are open, so that, $l_{C_0}$ is one trajectory connecting $P_0$ to $Q_1$ and lying completely in the negative half-space $\{Y\leq0\}$. This remark is very useful for the proof of Part (b), as shown below.

\medskip 

\noindent \textbf{Proof of Part (b)}. We proceed by induction on the number of oscillations with respect to the plane $\{Y=0\}$. In order to follow better the induction step, we first show how to pass from $K=1$ to $K=2$. To this end, assume that $p>p_s$, $p<2m-1$. Recalling the three sets introduced in \eqref{sets}, we split $\mathcal{C}$ into the following three disjoint sets:
\begin{equation*}
\begin{split}
&\mathcal{A}_1=\left\{C\in\mathcal{C}: l_C \ {\rm has \ at \ most \ one \ oscillation \ and \ there \ is} \ \eta_0\in\real, Y(\eta_0)<-\frac{N-2}{m}\right\},\\
&\mathcal{C}_1=\{C\in\mathcal{C}: l_C \ {\rm has \ at \ least \ two \ oscillations \ with \ respect \ to} \ \{Y=0\}\},\\
&\mathcal{B}_1=\mathcal{C}\setminus(\mathcal{A}_1\cup\mathcal{C}_1).
\end{split}
\end{equation*}
On the one hand, Lemma \ref{lem.globalQ3} implies that the orbit $l_C$ enters the stable node $Q_3$ for any $C\in\mathcal{A}_1$, and thus $\mathcal{A}_1$ is open. We next prove that $\mathcal{A}_1$ is non-empty. Recall that $C_0=\sup\mathcal{C}\in\mathcal{B}$, as explained at the end of the previous step and the trajectory $l_{C_0}$ connects $P_0$ to $Q_1$ fully contained in the half-space $\{Y\leq0\}$. The continuity with respect to the parameter $C$ on the unstable manifold of $P_0$ shows that there is $\delta>0$ such that the orbits $l_C$ for $C\in(C_0-\delta,C_0)$ enter the neighborhood $\mathcal{U}$ of $Q_1$ defined in \eqref{neigh.Q1} for which the local behavior near a generalized saddle \cite[Theorem 5.11]{Shilnikov} applies and, moreover, since $C_0=\sup\mathcal{C}$, there exists $C\in(C_0-\delta,C_0)\cap\mathcal{C}$. Lemma \ref{lem.w0} together with the definition of the set $\mathcal{C}$ in \eqref{sets} imply that the orbit $l_C$ with $C\in(C_0-\delta,C_0)\cap\mathcal{C}$ has exactly one oscillation with respect to the plane $\{Y=0\}$ which occurs before crossing the plane $\{Y=-2/(p-m)\}$, that is, $C\in\mathcal{A}_1$ and thus $\mathcal{A}_1$ is non-empty. On the other hand, since $p<2m-1$, the analysis performed in \cite[pp. 191-193]{S4} implies the non-emptiness of the set $\mathcal{C}_1$, while its mere definition and Lemma \ref{lem.tang} show that $\mathcal{C}_1$ is an open set. We then infer that $\mathcal{B}_1$ is non-empty and in particular $C_1:=\sup\mathcal{C}_1\in\mathcal{B}_1$ by the fact that both $\mathcal{A}_1$ and $\mathcal{C}_1$ are open sets. Thus, the trajectory $l_{C_1}$ has exactly one oscillation with respect to the plane $\{Y=0\}$ (since $C_1\notin\mathcal{C}_1$) and on this trajectory, $-(N-2)/m<Y(\eta)<0$ for $\eta\in(\eta_*,\infty)$, where $\eta_*=\sup\{\eta: Y(\eta)=0\}$. We infer from Lemma \ref{lem.globalQ1} that $l_{C_1}$ connects $P_0$ and $Q_1$ and has exactly one oscillation with respect to $\{Y=0\}$, as desired.

\medskip

This allows us to continue the induction step exactly in the same way. More precisely, letting $K\geq3$, $p>p_s$ such that $p<(Km-1)/(K-1)$, and assuming that we have already proved in the step $K$ the existence of at least $K-1$ different self-similar solutions in the form \eqref{SSS} (having exactly one, two,..., $K-1$ oscillations with respect to the plane $\{Y=0\}$), at the step $K+1$ we split the previous set $\mathcal{C}_{K-1}$ (of parameters $C$ for which the orbit $l_C$ has at least $K$ oscillations) into three sets as follows:
\begin{equation*}
\begin{split}
&\mathcal{A}_K=\left\{C\in\mathcal{C}_{K-1}: l_C \ {\rm has \ exactly} \ K \ {\rm oscillations \ and \ there \ is} \ \eta_0\in\real, Y(\eta_0)<-\frac{N-2}{m}\right\},\\
&\mathcal{C}_K=\{C\in\mathcal{C}_{K-1}: l_C \ {\rm has \ at \ least} \ K+1 \ {\rm oscillations}\},\\
&\mathcal{B}_K=\mathcal{C}_{K-1}\setminus(\mathcal{A}_K\cup\mathcal{C}_K).
\end{split}
\end{equation*}
Showing that $\mathcal{A}_{K}$ and $\mathcal{C}_K$ are both non-empty and open and then characterizing an element in $\mathcal{B}_K$ follows exactly in the same way as in the previous paragraph, and we omit the details which are identical to the ones above. We have plotted in Figure \ref{fig2} the outcome of some numerical experiments showing trajectories $l_C$ with different numbers of oscillations, as analyzed in the current proof. We also plot their profiles in the $(\xi,f)$ variables, where the oscillations are seen with respect to the constant profile $f(\xi)=k_*$ of the solution \eqref{const.sol}.

\begin{figure}[ht!]
  \begin{center}
  \subfigure[Various trajectories with different numbers of oscillations]{\includegraphics[width=7.5cm,height=6cm]{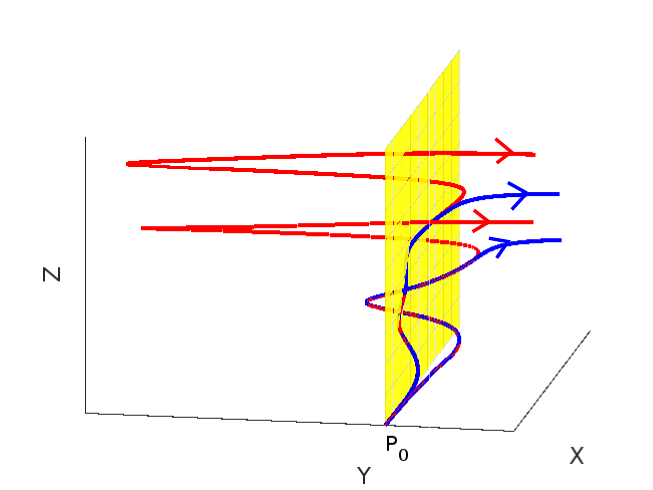}}
  \subfigure[Profiles corresponding to the previous trajectories]{\includegraphics[width=7.5cm,height=6cm]{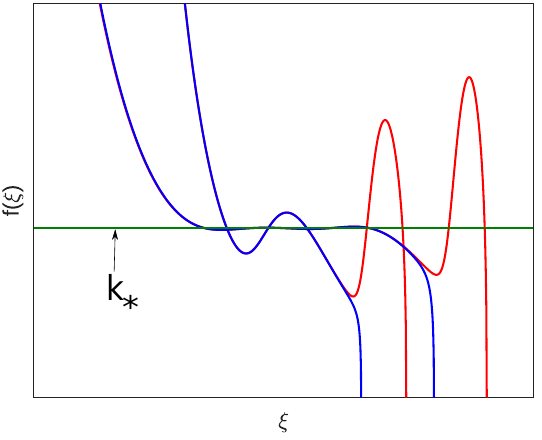}}
  \end{center}
  \caption{Various trajectories with oscillations, seen in the phase space and in profiles. Experiments for $m=2$, $N=100$, $p=2.2$, with $p_L\approx2.133$.}\label{fig2}
\end{figure}

Note in the end that solving with respect to $N$ the inequation $p_L=p_L(N)<(Km-1)/(K-1)$ leads to \eqref{LepK}, as claimed.
\end{proof}

\bigskip

\noindent \textbf{Acknowledgements} R. G. I. and A. S. are partially supported by the Project PID2020-115273GB-I00 and by the Grant RED2022-134301-T funded by MCIN/AEI/10.13039/ \\ 501100011033 (Spain).

\bibliographystyle{plain}

\begin{thebibliography}{1}

%
%
%
%
%
%
%
%
%
%
%
%
\bibitem{Carr}
J. Carr, \emph{Applications of Centre Manifold Theory}, Springer Verlag, New York, 1981.
%
%
%
%
%
%
%
\bibitem{FT00}
S. Filippas and A. Tertikas, \emph{On similarity solutions of a heat equation with a nonhomogeneous nonlinearity}, J. Differential Equations, \textbf{165} (2000), no. 2, 468-492.
%
\bibitem{Fu66}
H. Fujita, \emph{On the blowing up of solutions of the Cauchy problem for $u_t=\Delta u+u^{1+\alpha}$}, J. Fac. Sci. Univ. Tokyo Sect. I, \textbf{13} (1966), 109-124.
%

\bibitem{Ga86}
V. A. Galaktionov, \emph{Asymptotics of unbounded solutions of the nonlinear equation $u_t=(u^{\sigma}u_x)_x+u^{\beta}$ in a neighborhood of a singular point}, Soviet Math. Dokl., \textbf{33} (1986), 840-844.

\bibitem{Ga94}
V. A. Galaktionov, \emph{Blow-up for quasilinear heat equations with critical Fujita's exponents}, Proc. Royal Soc. Edinburgh, \textbf{124A} (1994), 517-525.

\bibitem{GKMS80}
V. A. Galaktionov, S. P. Kurdyumov, A. P. Mikhailov and A. A. Samarskii, \emph{Unbounded solutions of the Cauchy problem for the parabolic equation $u_t=\nabla(u^{\sigma}\nabla u)+u^{\beta}$}, Soviet Phys. Dokl., \textbf{25} (1980), 458-459.

\bibitem{GV97}
V. A. Galaktionov and J. L. V\'azquez, \emph{Continuation of blowup
solutions of nonlinear heat equations in several space dimensions},
Comm. Pure Appl. Math, \textbf{50} (1997), no. 1, 1-67.
%
\bibitem{GV02}
V. A. Galaktionov and J. L. V\'azquez, \emph{The problem of blow-up in nonlinear parabolic equations}, Discrete Cont. Dyn. Syst, \textbf{8} (2002), no. 2, 399-433.

\bibitem{GK85}
Y. Giga and R. Kohn, \emph{Asymptotically self-similar blow-up of semilinear heat equations}, Comm. Pure Appl. Math., \textbf{38} (1985), no. 3, 297-319.

\bibitem{GK87}
Y. Giga and R. Kohn, \emph{Characterizing blow-up using similarity variables}, Indiana Univ. Math. Journal, \textbf{36} (1987), 1-40.

\bibitem{GH}
J. Guckenheimer and Ph. Holmes, \emph{Nonlinear oscillation, dynamical systems and bifurcations of vector fields}, Applied Mathematical Sciences, vol. 42, Springer-Verlag, New York, 1990.
%
%
%
\bibitem{HV94}
M. A. Herrero and J. J. L. Vel\'azquez, \emph{Explosion des solutions des \'equations paraboliques semilin\'eaires supercritiques}, C. R. Acad. Sci. Paris S\'er I Math., \textbf{319(2)} (1994), 141-145.
%
%
%
\bibitem{ILS23}
R. G. Iagar, M. Latorre and A. S\'anchez, \emph{Blow-up patterns for a reaction-diffusion equation with weighted reaction in general dimension}, Adv. Differential Equations, \textbf{29} (2024), no. 7-8, 515-574.

\bibitem{IL13}
R. G. Iagar and Ph.~Lauren\c{c}ot, \emph{Existence and uniqueness of very singular solutions for a fast diffusion equation with gradient absorption},  J. London Math. Soc., \textbf{87} (2013), 509-529.

%
%
%
\bibitem{IMS23}
R. G. Iagar, A. I. Mu\~{n}oz and A. S\'anchez, \emph{Self-similar solutions preventing finite time blow-up for reaction-diffusion equations with singular potential}, J. Differential Equations, \textbf{358} (2023), 188-217.
%
\bibitem{IMS23b}
R. G. Iagar, A. I. Mu\~{n}oz and A. S\'anchez, \emph{Extinction and non-extinction profiles for the sub-critical fast diffusion equation with weighted source}, Submitted (2023), Preprint ArXiv no. 2302.09641.
%
%
%
%
\bibitem{IS21}
R. G. Iagar and A. S\'anchez, \emph{Blow up profiles for a quasilinear reaction-diffusion equation with weighted reaction}, J. Differential Equations, \textbf{272} (2021), no. 1, 560-605.

%
\bibitem{IS22}
R. G. Iagar and A. S\'anchez, \emph{Separate variable blow-up patterns for a reaction-diffusion equation with critical weighted reaction}, Nonlinear Anal., \textbf{217} (2022), Article ID 112740, 33p.
%
%
%
%

\bibitem{IS24}
R. G. Iagar and A. S\'anchez, \emph{Existence and multiplicity of blow-up profiles for a quasilinear diffusion equation with weighted source}, work in preparation

%

\bibitem{JL73}
D. D. Joseph and T. S. Lundgren, \emph{Quasilinear Dirichlet problems driven by positive sources}, Arch. Rational Mech. Anal., \textbf{49} (1972/73), 241-269.

\bibitem{Ka95}
T. Kawanago, \emph{Existence and behavior of solutions for $u_t=\Delta u^m+u^l$}, Adv. Math. Sci. and Appl., \textbf{7} (1997), no. 1, 367-400.

\bibitem{Le88}
L. A. Lepin, \emph{Countable spectrum of eigenfunctions of a nonlinear heat-conduction equation with distributed parameters}, Differential Equations, \textbf{24} (1988), 799-805.

\bibitem{Le89}
L. A. Lepin, \emph{The number of zeros in solutions of a second-order linear differential equation} (Russian), Latv. Gos. Univ., 1989, 37-46.

\bibitem{Le90}
L. A. Lepin, \emph{Self-similar solutions of a semilinear heat equation} (Russian), Mat. Model., \textbf{2} (1990), 63-74.
%
%
\bibitem{MM09}
H. Matano and F. Merle, \emph{Classification of type I and type II behaviors for a supercritical nonlinear heat equation}, J. Funct. Anal., \textbf{256} (2009), no. 4, 992-1064.

\bibitem{MM11}
H. Matano and F. Merle, \emph{Threshold and generic type I behaviors for a supercritical nonlinear heat equation}, J. Funct. Anal, \textbf{261} (2011), no. 3, 716-748.

\bibitem{Mi04}
N. Mizoguchi, \emph{Blowup behavior of solutions for a semilinear heat equation with supercritical nonlinearity}, J. Differential Equations, \textbf{205} (2004), 298-328.

\bibitem{Mi09}
N. Mizoguchi, \emph{Nonexistence of backward self-similar blowup solutions to a supercritical semilinear heat equation}, J. Funct. Anal., \textbf{257} (2009), 2911-2937.

\bibitem{Mi10}
N. Mizoguchi, \emph{On backward self-similar blow-up solutions to a supercritical semilinear heat equation}, Proc. Royal Soc. of Edinburgh, \textbf{140} (2010), 821-831. 


\bibitem{MS21}
A. Mukai and Y. Seki, \emph{Refined construction of Type II blow-up solutions for semilinear heat equations with Joseph-Lundgren supercritical nonlinearity}, Discrete Cont. Dynamical Systems, \textbf{41} (2021), no. 10, 4847-4885.
%
%
%
%
\bibitem{Pe}
L. Perko, \emph{Differential equations and dynamical systems. Third
edition}, Texts in Applied Mathematics, \textbf{7}, Springer Verlag,
New York, 2001.
%
%
%
\bibitem{Qi93}
Y.-W. Qi, \emph{On the equation $u_t=\Delta u^{\alpha}+u^{\beta}$}, Proc. Roy. Soc. Edinburgh Section A, \textbf{123} (1993), no. 2, 373-390.

%
\bibitem{QS}
P. Quittner, and Ph. Souplet, \emph{Superlinear parabolic problems. Blow-up, global existence and steady states}, Birkhauser Advanced Texts, Birkhauser Verlag, Basel, 2007.

\bibitem{S4}
A. A. Samarskii, V. A. Galaktionov, S. P. Kurdyumov, and A. P.
Mikhailov, \emph{Blow-up in quasilinear parabolic problems}, de
Gruyter Expositions in Mathematics, \textbf{19}, W. de Gruyter,
Berlin, 1995.
%
%
\bibitem{Shilnikov}
L. P. Shilnikov, A. Shilnikov, D. Turaev and L. O. Chua, \emph{Methods of qualitative theory in nonlinear dynamics. Part I}, World Scientific, 1998.
%
\bibitem{Sij}
J. Sijbrand, \emph{Properties of center manifolds}, Trans. Amer. Math. Soc., \textbf{289} (1985), no. 2, 431--469.

\bibitem{VPME}
J. L. V\'azquez, \emph{The porous medium equation. Mathematical theory}, Oxford Monographs in Mathematics, Oxford University Press, 2007.
%

\end{thebibliography}

\end{document}